\documentclass{article}
\usepackage[a4paper,margin=0.8in,top=3cm,bottom=4cm,footskip=0.5cm]{geometry}
\usepackage{pgfplots}
\usepackage{subcaption}
\usepackage{float}
\usepackage{pgfplots}
\usepackage{tikz}
\usetikzlibrary{fillbetween}
\usepgfplotslibrary{fillbetween}
\pgfplotsset{compat=1.17}
\usepackage{caption}
\usepackage{comment}
\usepackage[a4paper]{geometry}
\usepackage[utf8]{inputenc}
\usepackage{amsmath,amsfonts,amssymb,amsthm,bbm,mathrsfs}
\usepackage[shortlabels]{enumitem}
\usepackage{xcolor}
\usepackage{graphicx}
\usepackage{tikz}
\usepackage{appendix}
\usepackage{setspace}
\usepackage{fancyhdr}
\usepackage{hyperref}
\hypersetup{
    colorlinks=true,
    linkcolor=blue!60!cyan,
    filecolor=magenta,      
    urlcolor=blue!60!cyan,
    citecolor=blue!60!cyan,
    pdfpagemode=FullScreen,
}

\newtheorem{thm}{Theorem}[section]
\newtheorem{df}{Definition}[section]

\newtheorem{rmk}{Remark}[section]

\newtheorem{prop}{Proposition}[section]
\newtheorem{lm}{Lemma}[section]

\numberwithin{equation}{section}

\fancypagestyle{firstpage}{
\fancyhf{} 

\fancyfoot[L]{%
\footnotesize
\rule{1.8cm}{0.8pt}\\
$^\dagger$ Dipartimento di Matematica, Università degli Studi di Salerno,  Via Giovanni Paolo II, 132, 84084 Fisciano (SA), Italy.  \\
$^\ddagger$ Department of Mathematics, MAMCS Group, FST Errachidia, Moulay Ismail University of Meknes, P.O. Box 509, Boutalamine 52000, Errachidia, Morocco. \\
E-mail addresses: {\color{blue!60!cyan}\texttt{aelgrou@unisa.it}}, {\color{blue!60!cyan}\texttt{arhandi@unisa.it}}, {\color{blue!60!cyan}\texttt{j.salhi@umi.ac.ma}}
}
}

\title{\textbf{Long-Time Stability Analysis for Stochastic Evolution Equations with Multiplicative Noise}}

\author{Abdellatif Elgrou$^\dagger$,\, Abdelaziz Rhandi$^\dagger$ \,and\,  Jawad Salhi$^{\ddagger}$}

\date{}

\begin{document}

\maketitle
\vspace{-3em}

\thispagestyle{firstpage}
\begin{center}
\end{center}

\begin{abstract}
In this paper, we study the long-time stability behavior of a class of linear stochastic evolution equations in a Hilbert space with multiplicative noise. Explicit sufficient conditions for $p$-th moment and almost sure exponential stability are established, highlighting the interplay between the principal  eigenvalue of the governing operator, the drift coefficient, and the noise intensity. The relationship between these two notions of stability is also clarified. Applications to several stochastic partial differential equations are presented. In addition, a fully discrete spectral Galerkin method together with the implicit Euler--Maruyama scheme is shown to preserve these stability properties at the discrete level. Finally, numerical simulations are provided to confirm the theoretical results.
\end{abstract}
\maketitle
\smallskip


\section{Introduction}
Stochastic partial differential equations (SPDEs), and in particular stochastic evolution equations in infinite-dimensional Hilbert spaces, have become a central topic in modern probability theory and applied mathematics due to their ability to model systems subject to both spatial dynamics and random perturbations. Such models arise naturally in physics, biology, engineering, and finance, where uncertainty and noise play a fundamental role in the evolution of the system; see, for instance, \cite{Lue2021a,da2014stochastic,walsh1986,hu2019some,hu2015stochastic, hu2019some, mueller1991support} for classical foundations and recent developments.

In systems theory, the stability of stochastic evolution equations remains a prominent subject, having been greatly developed over the past years. Significant foundational developments were established by several authors (see, e.g., \cite{Haussmann1978,Caraballo1990,Khasminskii1994,ElBorai2003}), particularly regarding the asymptotic stability of the linear Ito equation in infinite dimension. Throughout this paper, we consider a class of linear stochastic evolution equations, although nonlinear cases have also been investigated; see, for instance, \cite{Ichikawa1982,Chow1982,Caraballo1994,Taniguchi1995,Caraballo1999,Ding2024}. In these works the stability of solutions has been studied with respect to sample paths or moments by using, for instance, coercivity conditions, Lyapunov functionals, and energy estimates. A common strategy in most of these papers consists of establishing almost sure exponential stability as a by-product of moment stability, which often necessitates overly restrictive assumptions to ensure the decay of expectations. We also refer to \cite{Liu2018}, where pathwise exponential stability is tackled by a direct approach, proving that a class of stochastic functional differential equations with delays driven by multiplicative noise, which are not exponentially stable in the moment sense, have the almost sure pathwise exponential stability property.

Building upon these classical results, the field has seen numerous advancements in recent years (see, e.g., \cite{Chow2011,Lv2022,Lv2023}), specifically addressing the impact of additive versus multiplicative noise on the stability of solutions.

In the literature, it is shown that multiplicative noise can act as a stabilizing mechanism for unstable deterministic systems, highlighting the intricate interplay between noise intensity and system dynamics (see, e.g., \cite{Caraballo2001, Mao2008}). This phenomenon, known as \emph{noise-induced stabilization}, has since been widely studied in both finite- and infinite-dimensional settings. In particular, the $p$-th moment and almost sure exponential stability results for stochastic heat equations driven by Brownian motion were established in \cite{xie2008moment}.

In parallel, the numerical analysis of evolutionary SPDEs has emerged as a highly active research area. This field has experienced rapid growth over recent decades, as reflected in an extensive body of literature (see, e.g., \cite{Anton2020, Jentzen2009, Jentzen2011, Kruse2014, Kruse2014Galerkin, Lord2014}).

From this perspective, there has been increasing interest in the development of numerical methods that preserve qualitative properties such as stability, decay rates, and asymptotic behavior. Classical results on mean-square, asymptotic, and moment exponential stability of numerical schemes for stochastic differential equations were established in \cite{Higham2000, Higham2007, Mao2015}, and later extended to infinite-dimensional systems driven by space-time noise in \cite{yang2025long,Lang2017}. For other related results on the numerical positivity and almost surely exponential stability of the stochastic heat equation, see for example \cite{yang2022stochastic}.

The present work provides a refined stability analysis for linear stochastic evolution equations in a Hilbert space with multiplicative noise, covering a broad class of stochastic partial differential equations. We also study a fully discrete approximation of our abstract continuous problem using a spectral Galerkin method combined with the implicit Euler--Maruyama scheme and show that it preserves the same stability properties at the discrete level, with numerical experiments supporting the theoretical results. Explicit sufficient conditions are derived for both $p$-th moment and almost sure exponential stability, showing that pathwise convergence may still hold even when moment stability fails. Within this abstract framework, the results recover classical cases such as those in \cite{xie2008moment} for the stochastic heat equation and extend to many previously unaddressed models, including stochastic biharmonic, fractional, and degenerate parabolic equations.

\subsection{Statement of the Problem}

We first introduce some notation and standing assumptions used throughout this paper. Let 
$(\Omega, \mathcal{F}, \mathbf{F}, \mathbb{P})$ be a complete filtered probability space, where 
$\mathbf{F} = \{\mathcal{F}_t\}_{t \ge 0}$ is the natural filtration generated by a one-dimensional standard Brownian motion 
$\{W(t)\}_{t \ge 0}$, augmented by all the $\mathbb{P}$-null sets in $\mathcal{F}$. We denote by $\mathbb{F}$ the progressive 
$\sigma$-algebra associated with the filtration $\mathbf{F}$, and by $\mathbb{E}(\cdot)$ the expectation operator with respect 
to the probability measure $\mathbb{P}$. Moreover, let $\mathbb{N} = \{0,1,2,\ldots\}$ denote the set of natural numbers including zero. 

Let $H$ be a Hilbert space endowed with inner product $\langle \cdot, \cdot \rangle_H$ and associated norm 
$\|\cdot\|_H$. We now introduce the following function spaces (see, e.g., \cite{Lue2021a,da2014stochastic} for further details).  Let $p \ge 1$.

\begin{itemize}

\item $L^p_{\mathcal{F}_t}(\Omega; H)$ denotes the space of all $\mathcal{F}_t$-measurable $H$-valued random variables $\xi$ such that $\mathbb{E}\big(\|\xi\|_H^p\big) < \infty$. It is a Banach space equipped with the norm
\[
\|\xi\|_{L^p_{\mathcal{F}_t}(\Omega; H)} = \big[\mathbb{E}(\|\xi\|_H^p)\big]^{1/p}.
\]
In the case $p=2$, it is a Hilbert space with the inner product
\[
\langle \xi_1,\xi_2\rangle_{L^2_{\mathcal{F}_t}(\Omega; H)}
= \mathbb{E}\big(\langle \xi_1,\xi_2\rangle_H\big).
\]

\item $L^p_{\mathbb{F}}(0,\infty; H)$ denotes the space of all $H$-valued $\mathbf{F}$-adapted processes $X(\cdot)$ such that $\mathbb{E}\left(\int_0^\infty \|X(t)\|_H^p\,dt\right) < \infty$. It is a Banach space endowed with the norm
\[
\|X\|_{L^p_{\mathbb{F}}(0,\infty; H)}
= \left(\mathbb{E}\int_0^\infty \|X(t)\|_H^p\,dt\right)^{1/p}.
\]
For $p=2$, it becomes a Hilbert space with the inner product
\[
\langle X_1,X_2\rangle_{L^2_{\mathbb{F}}(0,\infty; H)}
= \mathbb{E}\left(\int_0^\infty \langle X_1(t),X_2(t)\rangle_H\,dt\right).
\]

\item $L^p_{\mathbb{F}}(\Omega; C([0,\infty); H))$ denotes the space of all $H$-valued, $\mathbf{F}$-adapted continuous processes $X(\cdot)$ such that $\mathbb{E}\left(\sup_{t\ge 0} \|X(t)\|_H^p\right) < \infty$. It is a Banach space equipped with the norm
\[
\|X\|_{L^p_{\mathbb{F}}(\Omega; C([0,\infty); H))}
= \left[\mathbb{E}\left(\sup_{t\ge 0} \|X(t)\|_H^p\right)\right]^{1/p}.
\]

\end{itemize}

Let $A:\mathcal{D}(A)\subset H \to H$ be a linear self-adjoint, positive-definite operator with compact resolvent. Then, there exist sequences of eigenvalues and eigenfunctions $\{(\lambda_k,\phi_k)\}_{k=1}^{\infty}$ of $A$ such that $$ 0<\lambda_1 \leq \lambda_2 \leq \cdots,\quad \lambda_k \to +\infty \quad \text{as }\; k \to \infty, $$ 
and $\{\phi_k\}_{k=1}^{\infty}\subset\mathcal{D}(A)$ forms an orthonormal basis of $H$, with 
\begin{align}\label{spectdecA}
A\phi_k = \lambda_k \phi_k, \quad k \geq 1. 
\end{align}
Moreover, $-A$ generates an analytic $C_0$-semigroup $(S(t))_{t\ge0}$ on $H$ (see, e.g., \cite{engel2000one,lorenzi2021semigroups}), and $A$ satisfies the following spectral Poincaré-type inequality:
\begin{equation}\label{spectralpoincareinequality}
\langle A\varphi,\varphi\rangle_H \ge \lambda_1 \|\varphi\|_H^2, \quad \forall\, \varphi \in \mathcal{D}(A),
\end{equation}
where $\lambda_1 > 0$ denotes the principal eigenvalue of $A$, which can be given by
\begin{equation*}
\lambda_1 = \inf_{\varphi \in \mathcal{D}(A)\setminus\{0\}} 
\frac{\langle A\varphi,\varphi\rangle_H}{\|\varphi\|_H^2},
\end{equation*}
corresponding to the classical Rayleigh quotient associated with the operator $A$.

Consider the deterministic linear evolution equation
\begin{equation}\label{det_eq}
\begin{cases}
\frac{dy(t)}{dt}+ Ay(t) = \beta_0 y(t), \quad t>0,\\
y(0)=y_0,
\end{cases}
\end{equation}
where $\beta_0 \in \mathbb{R}$ and $y_0\in H$ is the initial condition. This system admits a unique weak solution $y\in C([0,\infty);H)$. Moreover, by differentiating $t\mapsto \|y(t)\|_H^2$, using \eqref{spectralpoincareinequality}, and applying Grönwall’s inequality, it follows that the solution of \eqref{det_eq} satisfies
\[
\|y(t)\|_H \le e^{-(\lambda_1-\beta_0)t}\|y_0\|_H,\quad \text{for any } t\ge0.
\]
This shows that the deterministic system \eqref{det_eq} is exponentially stable under the condition $\beta_0<\lambda_1$. For further details on exponential decay estimates for evolution equations, see, e.g., \cite{pazy2012semigroups}. However, in many applications the parameter $\beta_0$ is subject to uncertainties or environmental fluctuations, which may lead to loss of stability. On the other hand, random perturbations may also have a stabilizing effect on unstable systems. To this end, we assume that $\beta_0$ is affected by random fluctuations of the form $\beta_0 + \beta_1 \xi(t)$, where $\xi(t)$ is a rapidly varying random process representing external perturbations. A natural way to incorporate this is to model $\xi(t)$ by Gaussian white noise, formally written as $\dot{W}(t)=\frac{dW(t)}{dt}$, the generalized time derivative of the Brownian motion $W(t)$, and to replace $\xi(t)\,dt$ by the It\^o differential $dW(t)$. The system \eqref{det_eq} then becomes the following linear stochastic evolution equation with multiplicative noise
\begin{equation}\label{uncontrolled_for_eq}
\begin{cases}
dy(t) + Ay(t) \,dt = \beta_0 y(t)\,dt + \beta_1 y(t) \,dW(t), \quad t>0,\\
y(0)=y_0,
\end{cases}
\end{equation}
which will be the main object of study in this paper. Here $\beta_0,\beta_1 \in \mathbb{R}$ are given constants, and $y_0 \in L^p_{\mathcal{F}_0}(\Omega;H)$ for some $p \ge 1$ is a random initial condition.

We first recall the notions of mild and weak solutions of \eqref{uncontrolled_for_eq}; see, e.g., \cite{da2014stochastic,Lue2021a}.

\begin{df}
\begin{enumerate}
\item  An $H$-valued, $\mathbf{F}$-adapted, continuous stochastic process $y$ is called a \emph{mild solution} to \eqref{uncontrolled_for_eq} if, for every $t\geq0$, it holds that \[ y(t) = S(t)y_0 + \int_0^t S(t-s)\,\beta_0 y(s)\, ds + \int_0^t S(t-s)\,\beta_1 y(s)\, dW(s),\quad\textnormal{a.s.}\] 
\item An $H$-valued, $\mathbf{F}$-adapted, continuous stochastic process $y$ is called a \emph{weak solution} to \eqref{uncontrolled_for_eq}  if, for every $t\geq0$ and  $\phi \in \mathcal{D}(A)$, we have that
\[ \begin{aligned} \langle y(t), \phi \rangle_H &= \langle y_0, \phi \rangle_H - \int_0^t \langle y(s), A\phi \rangle_H \, ds \\ &\quad + \int_0^t \langle \beta_0 y(s), \phi \rangle_H \, ds + \int_0^t \langle \beta_1 y(s), \phi \rangle_H \, dW(s),\quad\textnormal{a.s.} \end{aligned} \] 
\end{enumerate}
\end{df} 
From \cite[Theorem 3.10]{Lue2021a}, the notions of weak and mild solutions to \eqref{uncontrolled_for_eq}  are equivalent. Hence, we refer to either of them simply as a solution to \eqref{uncontrolled_for_eq} . Moreover, since the operator $-A$ is self-adjoint and negative, it follows from \cite[Theorem 3.24]{Lue2021a} that the following well-posedness result holds.
\begin{prop}
    For any $y_0 \in L^p_{\mathcal{F}_0}(\Omega; H)$, where $p \geq 1$, the equation \eqref{uncontrolled_for_eq} admits a unique solution
    \[
    y \in L^p_{\mathbb{F}}(\Omega; C([0,\infty);H)) \cap L^p_{\mathbb{F}}(\Omega; L^2(0,\infty; \mathcal{D}(A^{1/2}))).
    \]
    Moreover, there exists a constant $C > 0$ such that
    \[
    \|y\|_{L^p_{\mathbb{F}}(\Omega; C([0,\infty);H))} + \|y\|_{L^p_{\mathbb{F}}(\Omega; L^2(0,\infty; \mathcal{D}(A^{1/2})))} \leq C \|y_0\|_{L^p_{\mathcal{F}_0}(\Omega; H)}.
    \]
\end{prop}

In this context, stability can be interpreted in several ways. Two commonly used notions are $p$-th moment exponential stability (in particular, mean-square exponential stability when $p=2$), which concerns the decay of the expected $p$-th power of the norm of the solution, and almost sure exponential stability, which describes the long-time behaviour of individual sample paths; see, e.g., \cite{Mao2008,Caraballo2001} for further details. In the sequel, we fix $p \ge 1$.

\begin{df}[$p$-th moment exponential stability]\label{def:ms_stability}
The solution $y$ of \eqref{uncontrolled_for_eq} is said to be $p$-th moment exponentially stable if there exist positive constants $C$ and $\mu$ such that for all $y_0 \in L^p_{\mathcal{F}_0}(\Omega; H)$,
\begin{equation*}
\mathbb{E} \left( \|y(t)\|_{H}^p \right) \leq C e^{-\mu t} \mathbb{E} \left( \|y_0\|_{H}^p \right), \quad \text{for all } t \geq 0.
\end{equation*}
\end{df}

\begin{df}[Almost sure exponential stability]\label{def:as_stability}
The solution $y$ of \eqref{uncontrolled_for_eq} is said to be almost surely exponentially stable in the $p$-sense if there exist a  deterministic constant $\delta > 0$ and a finite random variable $\xi$ such that for all $y_0 \in L^p_{\mathcal{F}_0}(\Omega; H)$,
\[
\|y(t)\|_H^p \le \xi e^{-\delta t}, \quad \text{for all } t \geq 0,\quad \text{a.s.}
\]
\end{df}

Clearly, if $y_0 = 0$ almost surely, then the solution of \eqref{uncontrolled_for_eq} satisfies $y(t)=0$ for all $t \ge 0$, almost surely, and the above stability properties hold trivially. Hence, without loss of generality, throughout this paper we assume that $y_0 \neq 0$ almost surely. Then, by Lemma \ref{solrep}, the solution satisfies $y(t)\neq 0$ for all $t \ge 0$, almost surely.

The main objectives of this paper are as follows. The first aim is to study the impact of multiplicative noise on the long-time dynamics of \eqref{uncontrolled_for_eq}, and to establish explicit sufficient conditions ensuring $p$-th moment exponential stability and almost sure exponential stability, showing that both notions of stability are determined by the balance between the principal eigenvalue $\lambda_1$ of the operator $A$, the drift coefficient $\beta_0$, and the noise intensity $\beta_1$. The second aim is to construct a spectral Galerkin approximation of \eqref{uncontrolled_for_eq} combined with an implicit Euler--Maruyama scheme and to prove that the stability properties are preserved at the discrete level. The third aim is to apply our theoretical results to some SPDEs and to provide numerical simulations illustrating the theoretical findings and confirming the stability behaviour of the system.

\begin{rmk}
Let $1\leq q \leq p$. Using the standard moment inequality
\[
\left[\mathbb{E}(\|y(t)\|_H^q)\right]^{1/q}
\le
\left[\mathbb{E}(\|y(t)\|_H^p)\right]^{1/p},
\]
it follows that $p$-th moment exponential stability implies $q$-th moment exponential stability. More precisely, if equation \eqref{uncontrolled_for_eq} is $p$-th moment exponentially stable, then
\[
\mathbb{E}(\|y(t)\|_H^q)
\le
C^{q/p} e^{-(\mu q/p)t}
\left[\mathbb{E}(\|y_0\|_H^p)\right]^{q/p}.
\]
In particular, the solution of \eqref{uncontrolled_for_eq} is also $q$-th moment exponentially stable (with a possibly different decay rate and constant).
\end{rmk}
\begin{rmk}
The notion of almost sure exponential stability in Definition \ref{def:as_stability} has been used in the literature; see, e.g., \cite{yang2014pth}. It shows that almost every trajectory decays exponentially with a deterministic rate $\delta > 0$, while the random variable $\xi=\xi(\omega)$ reflects the dependence on the sample path and the initial data.
\end{rmk}
\begin{rmk}
It is straightforward to extend the stability results of this paper, in particular Theorems \ref{pth_stability_abstract} and \ref{as_stability_abstract}, to the case where $W(t) = (W^1(t),\ldots,W^d(t))_{t\ge 0}$ is a $d$-dimensional standard Brownian motion. In this setting, equation \eqref{uncontrolled_for_eq} takes the form
\begin{equation}\label{genesys}
\begin{cases}
dy(t) + Ay(t)\,dt = \beta_0 y(t)\,dt + \displaystyle\sum_{i=1}^d \beta_i y(t)\,dW^i(t), \quad t>0,\\
y(0)=y_0.
\end{cases}
\end{equation}
The stability results for \eqref{genesys} can be derived analogously to those for \eqref{uncontrolled_for_eq}. More precisely, in the main stability conditions \eqref{pthmoexcon} (for $p$-th moment exponential stability) and \eqref{alconst} (for almost sure exponential stability), the term $\beta_1^2$ should be  replaced by $\sum_{i=1}^d \beta_i^2$. Hence, all components of the multiplicative noise jointly determine the stability properties of the system \eqref{genesys} through their aggregated intensity. 
\end{rmk}
\begin{rmk}
The stability of stochastic evolution equations driven by cylindrical or $Q$-Wiener processes are well studied in the literature; see, e.g., \cite{da2014stochastic,Taniguchi1995,Haussmann1978}. We also refer to \cite{Ding2024} for the case of multiplicative fractional Brownian motion. The present work focuses on the simpler scalar multiplicative noise setting, which already captures the essential interplay between the operator spectrum, the drift, and the noise intensity.
\end{rmk}

The rest of the paper is organized as follows. In Section \ref{sec2sec}, we present some preliminary lemmas. Section \ref{sec2} is devoted to the derivation of explicit conditions ensuring $p$-th moment and almost sure exponential stability for equation \eqref{uncontrolled_for_eq}, discuss the relation between these two notions of stability, and apply the results to some classes of SPDEs. In Section \ref{sec3}, we introduce a spectral decomposition and construct a finite-dimensional approximation of \eqref{uncontrolled_for_eq} via a spectral Galerkin method, together with an implicit Euler--Maruyama scheme, and show that the stability properties are preserved by the discrete system. Numerical simulations are also provided to validate the theoretical results.

\section{Some Preliminaries}\label{sec2sec}

In this section, we present two preliminary lemmas that will play a key role in the subsequent analysis. The following lemma follows from an Itô transformation argument.
\begin{lm}\label{solrep}
Let $y_0 \in L^p_{\mathcal F_0}(\Omega;H)$ and let $y$ be the solution of \eqref{uncontrolled_for_eq} with initial condition $y_0$. Then, for all $t\geq0$,
\[
y(t)=\exp\!\left(\beta_1 W(t)-\frac{\beta_1^2}{2}t\right)\, S(t)y_0,\quad \textnormal{a.s.},
\]
where $(S(t))_{t\geq0}$ is the $C_0$-semigroup generated by $-A+\beta_0 I$. In particular, 
$$y_0 \neq 0, \quad \textnormal{a.s.}
\;\; \Longrightarrow \;\;
y(t)\neq 0,
\quad \textnormal{for all } t\geq0,\quad \textnormal{a.s.}$$
\end{lm}

\begin{proof}
Define the process
\[
z(t)=\exp\!\left(-\beta_1 W(t)+\frac{\beta_1^2}{2}t\right)y(t).
\]
Applying It\^o's formula to $z$ and using that $y$ satisfies \eqref{uncontrolled_for_eq}, we obtain that $z$ satisfies the random evolution equation
\begin{align*}
\begin{cases}
\displaystyle \frac{dz(t)}{dt}=(-A+\beta_0 I)z(t), \quad t>0,\\
z(0)=y_0.
\end{cases}
\end{align*}
Hence,
\[
z(t)=S(t)y_0,
\]
which implies that
\[
y(t)=\exp\!\left(\beta_1 W(t)-\frac{\beta_1^2}{2}t\right)\,S(t)y_0,
\quad \textnormal{for all } t\geq0,\quad \textnormal{a.s.}
\]
From the spectral decomposition \eqref{spectdecA} of the operator \(A\), it follows immediately that \(S(t)\) is injective for every \(t \geq 0\). Moreover,
\[
\exp\!\left(\beta_1 W(t)-\frac{\beta_1^2}{2}t\right)>0.
\]
Hence, it follows that $y(t)\neq 0$ for all $t\ge 0$ almost surely whenever $y_0\neq 0$ almost surely.
\end{proof}

We now derive the following It\^o formula, which will play an important role in the derivation of our stability results.
\begin{lm}\label{lmm2itfomu}
Let $y$ be the solution of \eqref{uncontrolled_for_eq} with initial condition $y_0 \in L^p_{\mathcal{F}_0}(\Omega; H)$. Then, for all $t\geq0$, the following identity holds:
\begin{align}\label{deItofor}
\begin{aligned}
\|y(t)\|_H^p
&= \|y_0\|_H^p
- p \int_0^t \|y(s)\|_H^{p-2} \langle Ay(s), y(s)\rangle_H \, ds \\
&\quad + \left(p \beta_0 + \frac{p(p-1)}{2}\beta_1^2\right)
\int_0^t \|y(s)\|_H^p \, ds \\
&\quad + p \beta_1 \int_0^t \|y(s)\|_H^p \, dW(s),\qquad \textnormal{a.s.}
\end{aligned}
\end{align}
Equivalently, in differential form, identity \eqref{deItofor} can be written as follows:
\begin{align}\label{diffformitofo}
\begin{aligned}
d\|y(t)\|_H^p
&= -p \|y(t)\|_H^{p-2} \langle Ay(t), y(t)\rangle_H \, dt + p \beta_1 \|y(t)\|_H^p \, dW(t)\\
&\quad + \left(p \beta_0 + \frac{p(p-1)}{2}\beta_1^2\right)\|y(t)\|_H^p \, dt,\qquad \textnormal{a.s.}
\end{aligned}
\end{align}
\end{lm}

\begin{proof}
Let $t \in (0,T)$ and define $f(y(t))=\|y(t)\|_H^p$. By the classical differential form of Itô's formula (see, e.g., \cite[Theorem 2.139]{Lue2021a}), we have
\begin{equation}\label{Itoformlc}
df(y(t)) = D_xf(y(t))(dy(t)) + \frac{1}{2} D_x^2f(y(t))(dy(t),dy(t)),\qquad\textnormal{a.s.}
\end{equation}
By writing $f = h \circ g$, where
\[
h(s)=s^{p/2}, \qquad g(x)=\langle x,x\rangle_H,
\]
it is straightforward to see that
\begin{align}\label{dx1}
D_xf(y(t))(h)=p\|y(t)\|_H^{p-2}\langle y(t),h\rangle_H, \qquad \forall h\in H,
\end{align}
and
\begin{align}\label{dx2}
\begin{aligned}
D_x^2f(y(t))(h,k)
=&\, p(p-2)\|y(t)\|_H^{p-4}\langle y(t),h\rangle_H \langle y(t),k\rangle_H\\
&+ p\|y(t)\|_H^{p-2}\langle h,k\rangle_H,
\qquad \forall h,k\in H.
\end{aligned}
\end{align}
Hence, combining \eqref{dx1} and \eqref{dx2} with \eqref{Itoformlc}, it follows that
\begin{align}\label{Itoformlc2}
\begin{aligned}
df(y(t))
&= p\|y(t)\|_H^{p-2}\langle y(t),dy(t)\rangle_H \\
&\quad + \frac{p(p-2)}{2}\|y(t)\|_H^{p-4}|\langle y(t),dy(t)\rangle_H|^2 \\
&\quad + \frac{p}{2}\|y(t)\|_H^{p-2}\langle dy(t),dy(t)\rangle_H,\qquad\textnormal{a.s.}
\end{aligned}
\end{align}
Using that $y$ is the solution of \eqref{uncontrolled_for_eq} in \eqref{Itoformlc2}, we obtain
\begin{align}\label{Itoformlc2sec}
\begin{aligned}
df(y(t))
&= p\|y(t)\|_H^{p-2}\Big(-\langle y(t),Ay(t)\rangle_H dt+\beta_0 \|y(t)\|_H^{2}dt+\beta_1 \|y(t)\|_H^{2}dW(t)\Big)\\
&\quad + \frac{p(p-2)}{2}\beta_1^2\|y(t)\|_H^{p} dt + \frac{p}{2}\beta_1^2\|y(t)\|_H^{p} dt,\qquad\textnormal{a.s.}
\end{aligned}
\end{align}
Integrating \eqref{Itoformlc2sec} over $(0,t)$ gives the desired Itô formula \eqref{deItofor}. This completes the proof of Lemma \ref{lmm2itfomu}.
\end{proof}
\begin{rmk}
In what follows, instead of using the integral form of Itô's formula \eqref{deItofor}, we will typically use the differential form \eqref{diffformitofo}, as it simplifies the notation and reduces the complexity of many lengthy expressions.
\end{rmk}

\section{Exponential Stability of Equation \eqref{uncontrolled_for_eq}}\label{sec2}
In this section, we establish the $p$-th moment exponential stability and the almost sure exponential stability of equation \eqref{uncontrolled_for_eq}. We also investigate the relationship between these two notions of stability and present applications to several classes of SPDEs. 
\subsection{$p$-th Moment Exponential Stability}

In this subsection, we derive a sufficient condition on the parameters $\beta_0$, $\beta_1$, and $\lambda_1$ ensuring $p$-th moment exponential stability of the system \eqref{uncontrolled_for_eq}.
\begin{thm}\label{pth_stability_abstract}
Assume that the parameters $\beta_0$, $\beta_1$ and $\lambda_1$ satisfy the condition
\begin{align}\label{pthmoexcon}
 (p-1)\beta_1^2<2(\lambda_1 -  \beta_0).
 \end{align}
Then, for any initial condition $y_0 \in L^p_{\mathcal{F}_0}(\Omega; H)$, the unique solution $y$ of \eqref{uncontrolled_for_eq} is $p$-th moment exponentially stable. Specifically, the following estimate holds for all $t \geq 0$:
\begin{equation}\label{stability_res_abstract}
    \mathbb{E} \left( \|y(t)\|_H^p \right) \leq e^{-\mu_{p} t} \mathbb{E} \left( \|y_0\|_H^p \right),
\end{equation}
where the decay rate $\mu_{p}$ is given by
\begin{equation}\label{mu_def_abstract}
   \mu_{p} = p(\lambda_1 -  \beta_0) - \frac{p(p-1)}{2}\beta_1^2  > 0.
\end{equation}
\end{thm}

\begin{proof}
Let $t \geq 0$. By Itô's formula established in Lemma \ref{lmm2itfomu}, we have
\begin{align}\label{itf11}
\begin{aligned}
d\|y(t)\|_H^p
&= -p \|y(t)\|_H^{p-2} \langle Ay(t), y(t)\rangle_H \, dt + p \beta_1 \|y(t)\|_H^p \, dW(t)\\
&\quad + \left(p \beta_0 + \frac{p(p-1)}{2}\beta_1^2\right)\|y(t)\|_H^p \, dt,\quad \text{a.s.}
\end{aligned}
\end{align}
Next, we apply Itô's formula to the weighted process $t \mapsto e^{\mu_{p}t} \|y(t)\|_H^p$, where $\mu_p$ is the decay rate defined in \eqref{mu_def_abstract}. Using \eqref{itf11}, we obtain that
\begin{align*}
    d\left( e^{\mu_{p}t} \|y(t)\|_H^p \right) &= \mu_{p} e^{\mu_{p}t} \|y(t)\|_H^p dt + e^{\mu_{p}t} d\|y(t)\|_H^p \\
    &= \mu_{p} e^{\mu_{p}t} \|y(t)\|_H^p dt + e^{\mu_{p}t}\Big\{-p \|y(t)\|_H^{p-2} \langle Ay(t), y(t)\rangle_H \, dt \\
&\qquad+ p \beta_1 \|y(t)\|_H^p \, dW(t) + \left(p \beta_0 + \frac{p(p-1)}{2}\beta_1^2\right)\|y(t)\|_H^p \, dt\Big\}, \quad \text{a.s.}
\end{align*}
By the spectral Poincaré-type inequality \eqref{spectralpoincareinequality}, it follows that
\begin{align}\label{ineeabove16}
\begin{aligned}
    d\left( e^{\mu_{p}t} \|y(t)\|_H^p \right) &\leq e^{\mu_{p}t} \left( \mu_{p} - p\lambda_1 + p \beta_0 + \frac{p(p-1)}{2}\beta_1^2 \right) \|y(t)\|_H^p dt \\
    &\quad+ p \beta_1 \|y(t)\|_H^p \, dW(t), \quad \text{a.s.}
    \end{aligned}
\end{align}
Integrating \eqref{ineeabove16} over the time interval $[0, t]$ gives that
\begin{align}\label{ineqq2}
\begin{aligned}
    e^{\mu_{p}t} \|y(t)\|_H^p &\leq \|y_0\|_H^p + \int_0^t e^{\mu_{p}s} \left[ \mu_{p} - p\lambda_1 + p \beta_0 + \frac{p(p-1)}{2}\beta_1^2 \right] \|y(s)\|_H^p ds \\
    &\quad + p\beta_1 \int_0^t e^{\mu_{p}s} \|y(s)\|_H^p dW(s), \quad \textnormal{a.s.}
    \end{aligned}
\end{align}
Taking the expectation on both sides of \eqref{ineqq2} and noting that the stochastic integral is a martingale with zero mean, we obtain
\begin{align}\label{ineqq23}
    e^{\mu_{p}t} \mathbb{E}(\|y(t)\|_H^p) &\leq \mathbb{E}(\|y_0\|_H^p) + \left[ \mu_{p} - 
p(\lambda_1 -  \beta_0) + \frac{p(p-1)}{2}\beta_1^2 \right]  \mathbb{E}\bigg(\int_0^t e^{\mu_{p}s}\|y(s)\|_H^p ds\bigg).
\end{align}
According to the definition of the decay rate $\mu_{p}$ in \eqref{mu_def_abstract}, the second term on the right-hand side of \eqref{ineqq23} vanishes. Consequently, we obtain
$$
\mathbb{E}(\|y(t)\|_H^p) \leq e^{-\mu_{p}t} \mathbb{E}(\|y_0\|_H^p).
$$
This yields the desired stability estimate \eqref{stability_res_abstract}. This completes the proof of Theorem \ref{pth_stability_abstract}.
\end{proof}

\begin{rmk}\label{remarksonpmoments}
We now discuss some key aspects and consequences of Theorem \ref{pth_stability_abstract}.
\begin{itemize}
\item The $p$-th moment exponential stability result \eqref{stability_res_abstract} implies that the $p$-th moment of the solution decays exponentially fast. In particular, it follows that
\[
\gamma_p(y_0) := \limsup_{t \to \infty} \frac{1}{t} \log \mathbb{E}\,\big(\|y(t)\|_H^p\big) \leq -\mu_p < 0.
\]
In other words, the asymptotic logarithmic growth rate of the $p$-th moment, $\gamma_p(y_0)$, also called the $p$-th moment Lyapunov exponent of the solution (see, e.g., \cite{Mao2008}), is strictly negative.
\item The condition \eqref{pthmoexcon} reflects a balance between dissipation and stochastic forcing: stability holds when the effective dissipation $\lambda_1 - \beta_0$ dominates the noise intensity $\beta_1$, whose effect increases with $p$. The parameter $\beta_0$ shifts the damping, enhancing stability when negative and weakening it when positive. For $p=2$ (i.e., mean-square stability), or for general $p$ in the case of the stochastic heat equation, see, e.g., \cite{HernandezSantamaria2026,xie2008moment} and the references therein. The green region in Figure~\ref{fig:stable_colored} represents the set of all parameters $(\beta_1,\beta_0)\in\mathbb{R}^2$ satisfying the sufficient stability condition \eqref{pthmoexcon}, ensuring the moment exponential stability of \eqref{uncontrolled_for_eq}. Outside this region, stability is not guaranteed.
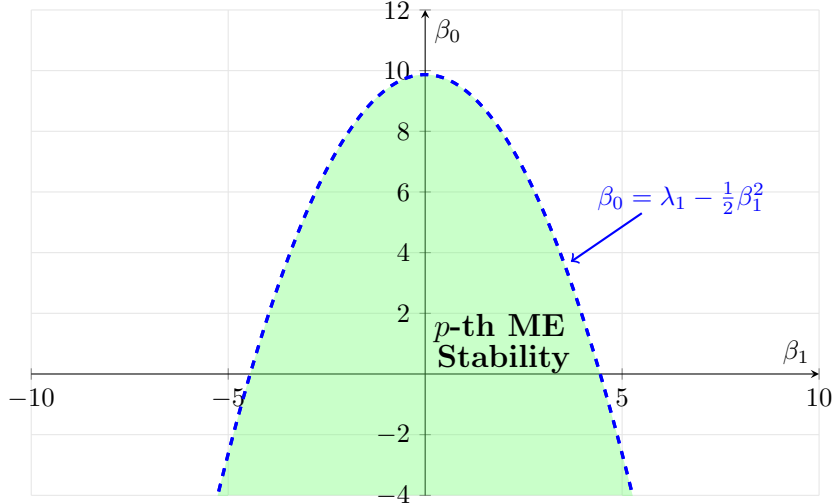
\begin{figure}[H]
\centering
\begin{tikzpicture}
\begin{axis}[
    width=12cm,
    height=8cm,
    xlabel={$\beta_1$},
    ylabel={$\beta_0$},
    xmin=-10, xmax=10,
    ymin=-4, ymax=12,
    axis lines=middle,
    grid=major,
    grid style={gray!20},
    xtick={-10,-5,0,5,10},
    ytick={-4,-2,0,2,4,6,8,10,12},
]
\pgfmathsetmacro{\lambdaone}{pi^2}
\addplot[
    name path=curve,
    blue,
    very thick,
    dashed,
    domain=-10:10,
    samples=200
] {\lambdaone - 0.5*x^2};
\addplot[name path=bottom, draw=none] {-4};
\addplot[
    green!60,
    opacity=0.35
] fill between[of=curve and bottom];
\addplot[
    name path=curveA,
    blue,
    very thick,
    dashed,
    domain=-10:10,
    samples=200
] {\lambdaone - 0.5*x^2};
\node[blue] at (axis cs:6.5,5.8)
{$\beta_0=\lambda_1-\frac{1}{2}\beta_1^2$};
\draw[blue,->,thick]
(axis cs:5.5,5.3) -- (axis cs:3.7,3.7);
\node at (axis cs:1.89,1.5) {\textbf{\large{$p$-th ME}}};
\node at (axis cs:2,0.5) {\textbf{\large{Stability}}};
\end{axis}
\end{tikzpicture}
\caption{$p$-th Moment exponential stability region in the $(\beta_1,\beta_0)$-plane for $p=2$ and $\lambda_1=\pi^2$.}
\label{fig:stable_colored}
\end{figure}
\item For fixed $p$, this can also be seen from Figure~\ref{fig:stable_colored}: as the principal eigenvalue $\lambda_1$ increases, the stability boundary shifts upward and the admissible parameter set $(\beta_1,\beta_0)$ enlarges, reflecting stronger dissipation that allows for larger noise intensity.

\item The condition \eqref{pthmoexcon} becomes more restrictive as \(p\) increases; higher moments therefore require stronger dissipation for stability. As shown in Figure~\ref{fig:p_effect}, the admissible region below the curve \(\beta_1 \mapsto \lambda_1 - \frac{p-1}{2}\beta_1^2\) shrinks as \(p\) grows. For \(p=1\), the noise contribution disappears and the condition reduces to \(\beta_0 < \lambda_1\), corresponding to the deterministic case. This is also reflected in the decay rate
\[
\mu_p = p(\lambda_1 - \beta_0) - \frac{p(p-1)}{2}\beta_1^2,
\]
where the linear stabilizing term \(p(\lambda_1-\beta_0)\) competes with the quadratic destabilizing noise term \(\frac{p(p-1)}{2}\beta_1^2\).
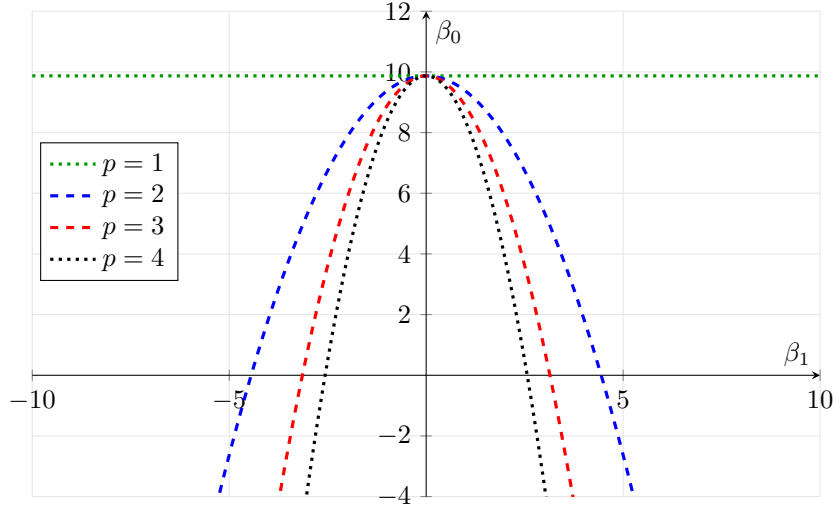
\begin{figure}[H]
\centering
\begin{tikzpicture}
\begin{axis}[
    width=12cm,
    height=8cm,
    xlabel={$\beta_1$},
    ylabel={$\beta_0$},
    xmin=-10, xmax=10,
    ymin=-4, ymax=12,
    axis lines=middle,
    grid=major,
    grid style={gray!20},
    xtick={-10,-5,0,5,10},
    ytick={-4,-2,0,2,4,6,8,10,12},
    legend style={at={(0.01,0.73)}, anchor=north west}
]
\pgfmathsetmacro{\lambdaone}{pi^2}
\addplot[
    name path=curveD,
    green!60!black,
    very thick,
     dotted,
    domain=-10:10,
    samples=200
] {\lambdaone};
\addlegendentry{$p=1$}
\addplot[
    name path=curveA,
    blue,
    very thick,
    dashed,
    domain=-10:10,
    samples=200
] {\lambdaone - 0.5*x^2};
\addlegendentry{$p=2$}
\addplot[
    name path=curveB,
    red,
    very thick,
    dashed,
    domain=-10:10,
    samples=200
] {\lambdaone - x^2};
\addlegendentry{$p=3$}
\addplot[
    name path=curveC,
    black,
    very thick,
    dotted,
    domain=-10:10,
    samples=200
] {\lambdaone - 1.5*x^2};
\addlegendentry{$p=4$}
\addplot[name path=bottom, draw=none] {-4};
\end{axis}
\end{tikzpicture}
\caption{Effect of increasing $p$ on the stability boundary, with $\lambda_1=\pi^2$.}
\label{fig:p_effect}
\end{figure}
\end{itemize}    
\end{rmk}

\subsection{Relation between $p$-th Moment and Almost Sure Exponential Stability}
In this subsection, we show that the $p$-th moment exponential stability of equation \eqref{uncontrolled_for_eq} is a stronger property than almost sure exponential stability. The proof builds on and extends ideas from \cite{Mao2008}, originally developed in the context of stochastic differential equations. This demonstrates that $p$-th moment exponential stability criteria may fail to capture the intrinsic stabilizing effect of multiplicative noise, which is instead revealed through almost sure stability analysis.
\begin{prop}\label{prop5}
The $p$-th moment exponential stability of equation \eqref{uncontrolled_for_eq} implies almost sure exponential stability.
\end{prop}

\begin{proof}
Assume that the $p$-th moment exponential stability of equation \eqref{uncontrolled_for_eq} holds. Then, there exist constants $C,\mu>0$ such that for any $y_0 \in L^p_{\mathcal{F}_0}(\Omega; H)$ and all $t \geq 0$,
\begin{equation}\label{pmoexpostaine}
    \mathbb{E} \left( \|y(t)\|_H^p \right) \leq C e^{-\mu t} \mathbb{E} \left( \|y_0\|_H^p \right).
\end{equation}
Fix an integer $n\geq 1$. By Itô's formula in Lemma \ref{lmm2itfomu}, we have for any $n-1\leq t\leq n$,
\begin{align*}
\begin{aligned}
\|y(t)\|_H^p
&= \|y(n-1)\|_H^p-\int_{n-1}^t p \|y(s)\|_H^{p-2} \langle Ay(s), y(s)\rangle_H \, ds\\
&\quad + \int_{n-1}^t \left(p \beta_0 + \frac{p(p-1)}{2}\beta_1^2\right)\|y(s)\|_H^p \, ds\\
&\quad+ \int_{n-1}^t p \beta_1 \|y(s)\|_H^p \, dW(s),\qquad \textnormal{a.s.}
\end{aligned}
\end{align*}
Using the spectral Poincaré-type inequality \eqref{spectralpoincareinequality}, we obtain
\begin{align}\label{ineqqas.sa}
\begin{aligned}
\|y(t)\|_H^p
&\leq \|y(n-1)\|_H^p+\kappa_1\int_{n-1}^t \|y(s)\|_H^p \, ds\\
&\quad+ \int_{n-1}^t p \beta_1 \|y(s)\|_H^p \, dW(s),\qquad \textnormal{a.s.},
\end{aligned}
\end{align}
where $\kappa_1=\lambda_1p+p|\beta_0|+\frac{p(p-1)}{2}\beta_1^2$. It then follows from \eqref{ineqqas.sa} that
\begin{align}\label{ineqqaf266}
\begin{aligned}
\mathbb{E}\Big(\sup_{n-1\leq t\leq n}\|y(t)\|_H^p\Big)
&\leq \mathbb{E}(\|y(n-1)\|_H^p)+\kappa_1\int_{n-1}^n \mathbb{E}\|y(s)\|_H^p \, ds\\
&\quad+ \mathbb{E}\bigg(\sup_{n-1\leq t\leq n}\bigg|\int_{n-1}^t p \beta_1 \|y(s)\|_H^p \, dW(s)\bigg|\bigg).
\end{aligned}
\end{align}
For the third term on the right-hand side of \eqref{ineqqaf266}, we apply the Burkholder–Davis–Gundy inequality (see, e.g., \cite[Theorem 2.140]{Lue2021a}). Then, there exists a positive constant $c$ such that
\begin{align*}
\begin{aligned}
\mathbb{E}\Big(\sup_{n-1\leq t\leq n}\|y(t)\|_H^p\Big)
&\leq \mathbb{E}(\|y(n-1)\|_H^p)+\kappa_1\int_{n-1}^n \mathbb{E}\|y(s)\|_H^p \, ds\\
&\quad+ c\,\mathbb{E}\bigg(\int_{n-1}^n p^2 \beta_1^2 \|y(s)\|_H^{2p} \, ds\bigg)^{1/2},
\end{aligned}
\end{align*}
which implies that
\begin{align}\label{ineqqbefyug}
\begin{aligned}
\mathbb{E}\Big(\sup_{n-1\leq t\leq n}\|y(t)\|_H^p\Big)
&\leq \mathbb{E}(\|y(n-1)\|_H^p)+\kappa_1\int_{n-1}^n \mathbb{E}\|y(s)\|_H^p \, ds\\
&\quad+ \mathbb{E}\bigg( \sup_{n-1\leq s\leq n}\|y(s)\|_H^{p}\int_{n-1}^n p^2 c^2\beta_1^2 \|y(s)\|_H^{p} \, ds\bigg)^{1/2}.
\end{aligned}
\end{align}
Applying the inequality $\sqrt{ab}\leq\frac{1}{2}(a+b)$ to the third term on the right-hand side of \eqref{ineqqbefyug}, we obtain
\begin{align*}
\begin{aligned}
\mathbb{E}\Big(\sup_{n-1\leq t\leq n}\|y(t)\|_H^p\Big)
&\leq \mathbb{E}(\|y(n-1)\|_H^p)+\kappa_1\int_{n-1}^n \mathbb{E}\|y(s)\|_H^p \, ds\\
&\quad+\frac{1}{2}\mathbb{E}\Big(\sup_{n-1\leq s\leq n}\|y(s)\|_H^p\Big)+ \frac{1}{2}\int_{n-1}^n p^2 c^2\beta_1^2 \mathbb{E}\|y(s)\|_H^{p} \, ds.
\end{aligned}
\end{align*}
This implies that
\begin{align*}
\begin{aligned}
\mathbb{E}\Big(\sup_{n-1\leq t\leq n}\|y(t)\|_H^p\Big)
&\leq 2\mathbb{E}(\|y(n-1)\|_H^p)+(2\kappa_1+p^2 c^2\beta_1^2) \int_{n-1}^n  \mathbb{E}\|y(s)\|_H^{p} \, ds.
\end{aligned}
\end{align*}
By \eqref{pmoexpostaine}, we obtain
\begin{align*}
\begin{aligned}
\mathbb{E}\Big(\sup_{n-1\leq t\leq n}\|y(t)\|_H^p\Big)
&\leq \Big(2C +(2\kappa_1+p^2c^2 \beta_1^2)\frac{1-e^{-\mu}}{\mu} \Big) e^{-\mu(n-1)}\mathbb{E}(\|y_0\|_H^p).
\end{aligned}
\end{align*}
Using the inequality \(1-e^{-a}\leq a\), we deduce that
\begin{align}\label{ineabomarine}
\begin{aligned}
\mathbb{E}\Big(\sup_{n-1\leq t\leq n}\|y(t)\|_H^p\Big)
&\leq \kappa_2 e^{-\mu(n-1)}\mathbb{E}(\|y_0\|_H^p),
\end{aligned}
\end{align}
where \(\kappa_2= 2C +2\kappa_1+p^2 c^2\beta_1^2\). By Markov’s inequality, for any $\varepsilon \in (0,\mu)$, we have
\[
\mathbb{P}\bigg(\sup_{n-1\leq t\leq n}\|y(t)\|_H^p > e^{-(\mu-\varepsilon)(n-1)} \bigg)
\leq e^{(\mu-\varepsilon)(n-1)} \mathbb{E}\Big(\sup_{n-1\leq t\leq n}\|y(t)\|_H^p\Big),
\]
which, together with \eqref{ineabomarine}, implies
\begin{align*}
\mathbb{P}\bigg(\sup_{n-1\leq t\leq n}\|y(t)\|_H^p > e^{-(\mu-\varepsilon)(n-1)} \bigg)
\leq \kappa_2 e^{-\varepsilon(n-1)}\mathbb{E}(\|y_0\|_H^p).
\end{align*}
By the Borel–Cantelli lemma, it follows that for almost all $\omega \in \Omega$, there exists an integer $n_0=n_0(\omega)\geq 1$ such that for all $n \geq n_0$, we have
\[
\sup_{n-1\leq t\leq n}\|y(t,\omega)\|_H^p \leq e^{-(\mu-\varepsilon)(n-1)}.
\]
Define the random time $T(\omega):=n_0(\omega)$. Then for all $t > T(\omega)$, there exists an integer $n$ such that $n-1\leq t\leq n$. Hence,
\[
\|y(t,\omega)\|_H^p \leq e^{-(\mu-\varepsilon)(n-1)}, \quad t>T(\omega).
\]
Since $n-1 \geq t-1$, we obtain
\[
\|y(t,\omega)\|_H^p \leq e^{-(\mu-\varepsilon)(t-1)}
= e^{\mu-\varepsilon}\, e^{-(\mu-\varepsilon)t}, \quad t>T(\omega).
\]
Now choose $\varepsilon=\frac{\mu}{2}$. Then
\[
\|y(t,\omega)\|_H^p \leq e^{\frac{\mu}{2}} e^{-\frac{\mu}{2}t}, \quad t>T(\omega).
\]
Define
\[
\xi(\omega):=\max\left\{e^{\frac{\mu}{2}},\ \sup_{0\leq t\leq T(\omega)} \left(e^{\frac{\mu}{2}t}\|y(t,\omega)\|_H^p\right)\right\}.
\]
Since the mapping $t \mapsto \|y(t,\omega)\|_H$ is continuous, it follows that $\xi(\omega)<\infty$. Therefore, for all $t\geq 0$, we obtain
\[
\|y(t,\omega)\|_H^p \leq \xi(\omega)\, e^{-\frac{\mu}{2} t}.
\]
This implies that equation \eqref{uncontrolled_for_eq} is almost surely exponentially stable, which completes the proof of Proposition \ref{prop5}.
\end{proof}
\begin{rmk}
In Proposition \ref{prop5}, we proved that $p$-th moment exponential stability implies almost sure exponential stability for equation \eqref{uncontrolled_for_eq}. The reverse implication does not hold in general, since almost sure stability is a pathwise property concerning $\|y(t,\omega)\|_H$ for almost every fixed $\omega \in \Omega$, whereas $p$-th moment exponential stability depends on the averaged quantity $\mathbb{E}(\|y(t)\|_H^p)$ over all $\omega \in \Omega$. Thus, rare but large fluctuations may not affect almost sure decay, but can prevent decay of moments.
\end{rmk}

\subsection{Almost Sure Exponential Stability}
Although $p$-th moment stability ensures almost sure exponential stability, it requires the same conditions on the parameters. To expand the stability domain and capture noise-induced stabilization, we now derive a sharper sufficient condition on $\beta_0$, $\beta_1$, and $\lambda_1$ for the almost-sure exponential stability of system \eqref{uncontrolled_for_eq}.

\begin{thm}\label{as_stability_abstract}
Assume that the parameters $\beta_0$, $\beta_1$, and $\lambda_1$ satisfy the condition  
\begin{align}\label{alconst}
\beta_1^2 > 2(\beta_0 - \lambda_1). 
\end{align}
Then, there exists a finite random variable $\xi$ such that for all $y_0 \in L^p_{\mathcal{F}_0}(\Omega; H)$, the solution $y$ of \eqref{uncontrolled_for_eq} satisfies that
\begin{equation}\label{as_stability_res_abstract}
\|y(t)\|_H^p \le \xi\, e^{-\frac{\mu_{as}}{2} t}, \quad \text{for all } t \geq 0,\quad \textnormal{a.s.},
\end{equation}
where the almost sure decay rate is given by
\begin{equation}\label{mu_as_def_abstract}
\mu_{as} = \frac{p}{2}\beta_1^2 + p(\lambda_1 - \beta_0) > 0.
\end{equation}
\end{thm}
\begin{proof}
Let $t \geq 0$ and let $\eta > 0$ be a parameter to be chosen later in the proof. We define the auxiliary stochastic process
$$
S_p(t) := \|y(t)\|_H^p + e^{-\eta t}, \quad \text{for all } t \geq 0.
$$
Applying Itô's formula to the process $\log S_p(t)$, we obtain
$$
d\log S_p(t) = \frac{1}{S_p(t)} dS_p(t) -\frac{1}{2S^2_p(t)}(dS_p(t))^2, \quad \text{a.s.}, 
$$
where
$$
dS_p(t)=-\eta e^{-\eta t}dt+d\|y(t)\|^p_H, \quad \text{a.s.} 
$$
Now, using Itô's formula from Lemma \ref{lmm2itfomu}, we deduce that
\begin{align}\label{equali188}
\begin{aligned}
   d\log S_p(t) =& \frac{1}{S_p(t)} \bigg[-\eta e^{-\eta t} -p \|y(t)\|_H^{p-2} \langle Ay(t), y(t)\rangle_H \,  \\
&\hspace{1.5cm} + \left(p \beta_0 + \frac{p(p-1)}{2}\beta_1^2 \right)\|y(t)\|_H^p \, \bigg] dt \\
& 
 -\frac{p^2\beta_1^2}{2S^2_p(t)}\|y(t)\|^{2p}_H dt+ \frac{p \beta_1}{S_p(t)} \|y(t)\|_H^p \, dW(t), \quad \text{a.s.} 
\end{aligned}
\end{align}
Integrating \eqref{equali188} over the time interval $(0,t)$ and using the Poincar\'e inequality \eqref{spectralpoincareinequality}, we obtain
\begin{align}\label{logS_int_abstract} 
\begin{aligned}
\log S_p(t) &\leq \log (\|y_0\|_H^p + 1) -\frac{p^2\beta_1^2}{2}\int_0^t \frac{\|y(s)\|^{2p}_H}{S^2_p(s)} ds + M_p(t)\\
&\quad+ \int_0^t \frac{1}{S_p(s)}\bigg[-\eta e^{-\eta s} + \bigg( p \beta_0 + \frac{p(p-1)}{2}\beta_1^2 -p\lambda_1\bigg)\|y(s)\|_H^p\bigg] ds, \quad \text{a.s.},
\end{aligned}
\end{align}
where
$$
M_p(t) := \int_0^t \frac{p \beta_1}{S_p(s)} \|y(s)\|_H^p \, dW(s),
$$
is a continuous martingale with initial value $M_p(0)=0$, whose quadratic variation is given by
\[
[M_p]_t = \int_0^t \frac{p^2\beta_1^2}{S_p^2(s)} \|y(s)\|_H^{2p} ds \leq p^2\beta_1^2 t < +\infty.
\]
By the exponential martingale inequality (see, e.g., \cite{Mao2008}), we have that for any positive numbers $T$, $\varepsilon$, and $k$,
$$
\mathbb{P}\Big(\max_{0\leq t\leq T}\Big[M_p(t)-\frac{\varepsilon}{2}[M_p]_t\Big]> k\Big)\leq e^{-\varepsilon k}.
$$
We now choose an arbitrary $\varepsilon \in (0,1/2p)$ and set $k=\frac{2\log N}{\varepsilon}$, where $N \geq 1$ is an integer. Then, it follows that
$$
\mathbb{P}\bigg(\max_{0\leq t\leq T}\Big[M_p(t)-\frac{\varepsilon}{2}[M_p]_t\Big]> \frac{2\log N}{\varepsilon}\bigg)\leq \frac{1}{N^2}.
$$
By the Borel–Cantelli lemma, it follows that for almost all $\omega\in\Omega$, there exists an integer $N_0=N_0(\omega)\geq 1$ such that for all $N \geq N_0$ and all $t \in [0,N]$, we have
$$
M_p(t) \leq \frac{\varepsilon}{2} [M_p]_t + \frac{2}{\varepsilon} \log N.
$$
From \eqref{logS_int_abstract}, we deduce that for all $t \in [0,N]$ with $N \geq N_0$,
\begin{align}\label{logS_int_abstractsec} 
\begin{aligned}
\log S_p(t) &\leq \log (\|y_0\|_H^p + 1) -\frac{p^2}{2}\beta_1^2(1-\varepsilon)\int_0^t \frac{\|y(s)\|^{2p}_H}{S^2_p(s)} ds  + \frac{2}{\varepsilon} \log N\\
&\quad+ \int_0^t \frac{1}{S_p(s)}\bigg[-\eta e^{-\eta s} + \bigg( p \beta_0 + \frac{p(p-1)}{2}\beta_1^2 -p\lambda_1\bigg)\|y(s)\|_H^p\bigg] ds, \quad \text{a.s.} 
\end{aligned}
\end{align}
Let us introduce the auxiliary function $z(s) = e^{\eta s}\|y(s)\|_H^p$. Then, inequality \eqref{logS_int_abstractsec} can be expressed as follows
\begin{equation}\label{G_ineq_abstract} \log S_p(t) \leq \log (\|y_0\|_H^p + 1) + \int_0^t G(z(s)) ds + \frac{2}{\varepsilon} \log N, \quad \text{a.s.}, 
\end{equation}
where the function $G: \mathbb{R}_+ \to \mathbb{R}$ is defined by
\begin{equation*}
G(u) := \frac{-\eta + A u}{u+1} - \frac{B u^2}{(u+1)^2},
\end{equation*}
with
$$
A=p \beta_0 + \frac{p(p-1)}{2}\beta_1^2 -p\lambda_1
\quad \text{and} \quad
B=\frac{p^2}{2}\beta_1^2(1-\varepsilon).
$$
We now set $\eta = 2B - A$. Since $\varepsilon \in (0,1/2p)$, it follows that
$$
\eta = p^2\beta_1^2(1-\varepsilon)-\frac{p(p-1)}{2}\beta_1^2 - p\beta_0 + p\lambda_1
\geq \frac{p^2}{2}\beta_1^2 - p\beta_0 + p\lambda_1.
$$
Recalling the condition \eqref{alconst}, this ensures that the parameter $\eta > 0$. On the other hand, it is straightforward to verify that
$$
G'(u) = \frac{(A+\eta)(u+1)-2Bu}{(u+1)^3} = \frac{2B}{(u+1)^3} > 0, \quad \text{for all } u \in \mathbb{R}_+.
$$
Hence, $G$ is increasing, and therefore, for all $u \in \mathbb{R}_+$, we have
\begin{align}\label{ineqGu}
G(u)\leq \lim_{v \to \infty} G(v) = A-B.
\end{align} 
Combining \eqref{G_ineq_abstract} and \eqref{ineqGu}, we obtain that, for all $t \in [0,N]$ with $N \geq N_0$,
\begin{equation}\label{G_ineq_abstractsec} 
\log S_p(t) \leq \log \bigl(\|y_0\|_H^p + 1\bigr) + t(A - B) + \frac{2}{\varepsilon} \log N, \quad \text{a.s.} 
\end{equation}
Thus, for any $t \in [N-1, N]$, dividing \eqref{G_ineq_abstractsec} by $t$ and letting $t \to \infty$, we get
\begin{align*} 
\limsup_{t \to \infty} \frac{1}{t} \log S_p(t) 
&\leq A - B \\ 
&= p\beta_0 + \frac{p(p-1)}{2}\beta_1^2 - p\lambda_1 - \frac{p^2}{2}\beta_1^2(1 - \varepsilon), \quad \text{a.s.} 
\end{align*}
Letting $\varepsilon \to 0^+$, we finally obtain
\begin{equation*} 
\limsup_{t \to \infty} \frac{1}{t} \log S_p(t) \leq -\Big(\frac{p}{2} \beta_1^2+p(\lambda_1-\beta_0)\Big)=-\mu_{as}, \quad \text{a.s.}
\end{equation*} 
Then, for almost every $\omega \in \Omega$ and for every $\varepsilon > 0$, there exists $T(\omega) > 0$ such that
\[
\frac{1}{t} \log S_p(t,\omega) \leq -\mu_{as} + \varepsilon,
\quad \text{for all } t > T(\omega).
\]
Set $\varepsilon = \frac{\mu_{as}}{2}$. Then there exists $T(\omega) > 0$ such that
\[
\frac{1}{t} \log S_p(t,\omega) \leq -\frac{\mu_{as}}{2},
\quad \text{for all } t > T(\omega).
\]
Hence,
\[
S_p(t,\omega) \leq e^{-\frac{\mu_{as}}{2} t}, 
\quad \text{for all } t > T(\omega).
\]
Since $\|y(t,\omega)\|_H^p \leq S_p(t,\omega)$, it follows that
\[
\|y(t,\omega)\|_H^p \leq e^{-\frac{\mu_{as}}{2} t},
\quad \text{for all } t > T(\omega).
\]
Now define
\[
\xi(\omega) := \max\left\{1,\; \sup_{0 \leq t \leq T(\omega)} 
\left(e^{\frac{\mu_{as}}{2} t}\|y(t,\omega)\|_H^p\right) \right\}.
\]
Since the mapping $t \mapsto \|y(t,\omega)\|_H$ is continuous, it follows that $\xi(\omega) < \infty$. Therefore, for all $t \geq 0$, we obtain
\[
\|y(t,\omega)\|_H^p \leq \xi(\omega) e^{-\frac{\mu_{as}}{2} t}.
\]
This completes the proof of Theorem \ref{as_stability_abstract}.
\end{proof}

\begin{rmk}
We now present some remarks on the almost sure exponential stability result.
\begin{itemize}
\item Under condition \eqref{alconst}, Theorem \ref{as_stability_abstract} implies that almost all trajectories converge exponentially to zero. In particular, the solution of \eqref{uncontrolled_for_eq} is almost surely exponentially stable in the Lyapunov sense (see, e.g., \cite{Mao2008}), and its long-time growth rate satisfies
\[
\alpha_p(y_0):=\limsup_{t\to\infty}\frac{1}{t}\log \|y(t,\omega)\|_H^p \le -\frac{\mu_{as}}{2}, \quad \textnormal{a.s.}
\]
Moreover, \eqref{mu_as_def_abstract} shows that the decay of $\|y(t)\|_H^p$ is proportional to $p$, so larger values of $p$ yield faster decay.

\item The multiplicative noise term $\beta_1 y\, dW(t)$ has a dual effect: it may destabilize the system in the $p$-th moment sense, while improving the almost sure decay rate through the term $\beta_1^2$ in \eqref{mu_as_def_abstract}, thereby enhancing pathwise stability. In particular, even when the deterministic system is unstable (i.e., \(\lambda_1 \le \beta_0\)), sufficiently large noise intensity satisfying
$|\beta_1|>\sqrt{2(\beta_0-\lambda_1)}$ can induce almost sure exponential stability, illustrating the phenomenon of stochastic stabilization. This effect enlarges the admissible parameter region \((\beta_1,\beta_0)\) for which condition \eqref{alconst} holds, as shown by the green region in Figure~\ref{fig:as_stability}.
\begin{figure}[H]
\centering
\begin{tikzpicture}
\begin{axis}[
    width=13cm,
    height=8.5cm,
    xlabel={$\beta_1$},
    ylabel={$\beta_0$},
    xmin=-6, xmax=6,
    ymin=-2, ymax=12,
    axis lines=middle,
    grid=major,
    grid style={gray!15},
    xtick={-5,-4,-2,0,2,4,5},
    ytick={-2,0,2,4,6,8,10,12},
    samples=200,
]
\def\lambdaone{1}
\addplot[
    name path=curve,
    blue,
    very thick,
     dashed,
    smooth,
    domain=-5:5
]
{\lambdaone + 0.5*x^2};
\addplot[name path=bottom, draw=none] {-2};
\addplot[
    green!60,
    opacity=0.35
] fill between[of=curve and bottom];
\addplot[
    blue,
    very thick,
     dashed
]
{\lambdaone + 0.5*x^2};
\node[blue] at (axis cs:1.5,7.5)
{$\beta_0=\lambda_1+\frac{1}{2}\beta_1^2$};
\draw[blue,->,thick]
(axis cs:1.3,7.1) -- (axis cs:2.9,5.5);
\node at (axis cs:3.4,1.8)
{\textbf{{\large ASE Stability}}};
\end{axis}
\end{tikzpicture}
\caption{Almost sure exponential stability region, with $\lambda_1=1$.}
\label{fig:as_stability}
\end{figure}
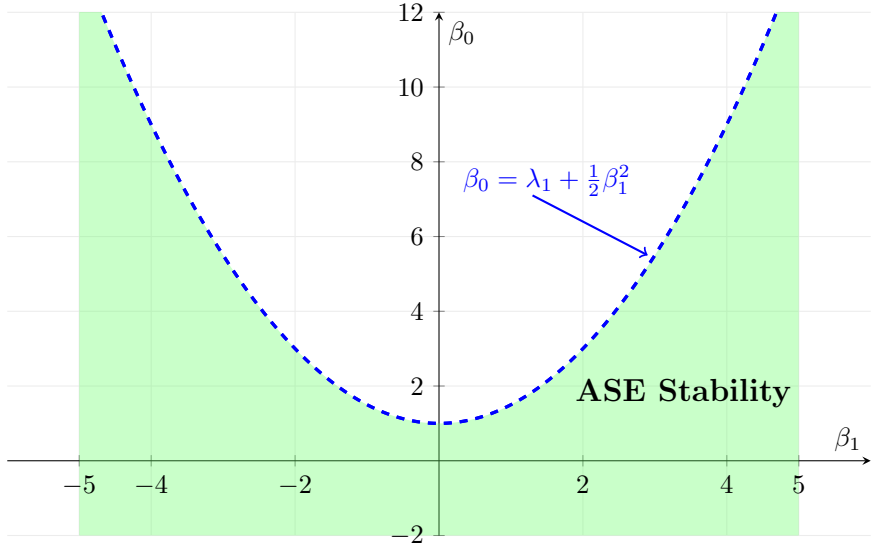
\end{itemize}
\end{rmk}

\subsection{Illustrative Examples for Stochastic PDEs}

In this subsection, we illustrate the abstract results on $p$-th moment exponential stability (Theorem \ref{pth_stability_abstract}) and almost sure exponential stability (Theorem \ref{as_stability_abstract}) by applying them to some classical classes of stochastic partial differential equations, including the stochastic heat equation, the stochastic biharmonic heat equation, the fractional stochastic heat equation, and the stochastic degenerate parabolic equation, all associated with operators satisfying the standing assumptions. In each case, we verify that the corresponding operator fits the abstract framework and recall the explicit stability conditions, analyzing them in relation to the structure of each equation. For further details on the semigroup framework, spectral properties of the underlying operator, and the well-posedness of stochastic evolution equations, we refer the reader to \cite{da2014stochastic,Lue2021a,engel2000one,prevot2007concise}. Throughout the examples, we assume that $G \subset \mathbb{R}^m$ ($m \geq 1$) is a bounded domain with smooth boundary $\partial G$, and we set $H = L^2(G)$.
\subsubsection{Stochastic Heat Equation}
We consider the following system
\begin{equation}\label{SHEquation}
\begin{cases}
dy - \Delta y\,dt = \beta_0 y\,dt + \beta_1 y\,dW(t), & (t,x) \in (0,\infty)\times G,\\
y = 0, & (t,x) \in (0,\infty)\times \partial G,\\
y(0,\cdot) = y_0, & x \in G,
\end{cases}
\end{equation}
where $y_0 \in L^p_{\mathcal{F}_0}(\Omega; L^2(G))$, $\beta_0, \beta_1 \in \mathbb{R}$, and $p \ge 1$. Define the operator $A = -\Delta$ with domain $\mathcal{D}(A)=H^2(G)\cap H_0^1(G)$. Then $A$ is self-adjoint, positive-definite, and has compact resolvent, so that the abstract setting applies. Let $\{\lambda_k\}_{k=1}^\infty$ denote the eigenvalues of $A$, satisfying
\[
0 < \lambda_1 \le \lambda_2 \le \cdots \to \infty.
\]

Applying Theorem \ref{pth_stability_abstract}, the solution
$$
y \in L^p_{\mathcal{F}}(\Omega; C([0,\infty); L^2(G))) \cap L^p_{\mathcal{F}}(\Omega; L^2(0,\infty; H_0^1(G)))
$$
of \eqref{SHEquation} is $p$-th moment exponentially stable provided that
$(p-1)\beta_1^2 < 2(\lambda_1 - \beta_0)$,
while Theorem \ref{as_stability_abstract} ensures almost sure exponential stability under the condition $\beta_1^2 > 2(\beta_0 - \lambda_1)$. These conditions highlight the balance between the dissipative effect of the Laplacian (through $\lambda_1$), the deterministic growth rate $\beta_0$, and the intensity of the multiplicative noise $\beta_1$. We note that similar stability results for the stochastic heat equation have been established in the literature; see, e.g., \cite{HernandezSantamaria2026,xie2008moment,yang2022stochastic}.

\subsubsection{Stochastic Biharmonic Heat Equation}
We introduce the following system
\begin{equation}\label{BihSHEquation}
\begin{cases}
dy + \Delta^2 y\,dt = \beta_0 y\,dt + \beta_1 y\,dW(t), & (t,x) \in (0,\infty)\times G,\\
\text{BC}, & \\
y(0,\cdot) = y_0, & x \in G,
\end{cases}
\end{equation}
where $y_0 \in L^p_{\mathcal{F}_0}(\Omega; L^2(G))$, $p \ge 1$, and we consider one of the two classical types of boundary conditions (BC):

\begin{itemize}
\item First, under the hinged boundary conditions
\[
y = \Delta y = 0,\qquad (t,x)\in(0,\infty)\times \partial G,
\]
the biharmonic operator $A=\Delta^2$ can be defined as $A = (-\Delta)^2$, where $-\Delta$ denotes the Dirichlet Laplacian on $G$. Its domain is given by
\[
\mathcal{D}(A) = \{ y \in H^4(G) \mid y, \Delta y \in  H_0^1(G)\}.
\]
In this case, $A$ reduces spectrally to the square of the Dirichlet Laplacian, sharing the same eigenfunctions, while the eigenvalues satisfy $\mu_k = \lambda_k^2$, where $\{\lambda_k\}_{k\ge1}$ are the eigenvalues of the Dirichlet Laplacian. Moreover, the operator $A$ is self-adjoint, positive-definite, and has compact resolvent, with fractional domain $\mathcal{D}(A^{1/2}) = H^2(G)\cap H_0^1(G)$.

\item Second, under the clamped boundary conditions
\[
y = \partial_\nu y = 0,\qquad (t,x)\in(0,\infty)\times \partial G,
\]
the biharmonic operator is endowed with the domain $\mathcal{D}(A)= H^4(G)\cap H^2_0(G)$, and is no longer reducible to the Dirichlet Laplacian. Nevertheless, it remains self-adjoint, positive definite, and has compact resolvent, with $\mathcal{D}(A^{1/2}) = H^2_0(G)$.
\end{itemize}
In both cases, denoting by $\{\mu_k\}_{k=1}^\infty$ the eigenvalues of $A$, we have
\[
0 < \mu_1 \le \mu_2 \le \cdots \to \infty.
\]
Hence, as in the first example, we apply Theorems \ref{pth_stability_abstract} and \ref{as_stability_abstract} to derive stability results for equation \eqref{BihSHEquation}. In particular, the solution 
$$y \in L^p_{\mathcal F}(\Omega; C([0,\infty); L^2(G))) \cap L^p_{\mathcal F}(\Omega; L^2(0,\infty; \mathcal{D}(A^{1/2})))$$
of \eqref{BihSHEquation} is $p$-th moment exponentially stable whenever $(p-1)\beta_1^2 < 2(\mu_1 - \beta_0)$, while Theorem \ref{as_stability_abstract} guarantees almost sure exponential stability under the condition $\beta_1^2 > 2(\beta_0 - \mu_1)$. While the fourth-order operator defined with clamped boundary conditions possesses a distinct spectral structure, its counterpart under hinged conditions exhibits a direct spectral relationship with the Laplacian. This specific hinged case provides stronger coercivity than standard second-order diffusion, yielding enhanced intrinsic damping and superior stability margins under multiplicative noise.


\subsubsection{Fractional Stochastic Heat Equation}
We consider the system
\begin{equation}\label{FracSHEquation}
\begin{cases}
dy + (-\Delta)^{s} y\,dt = \beta_0 y\,dt + \beta_1 y\,dW(t), & (t,x) \in (0,\infty)\times G,\\
y = 0, & (t,x) \in (0,\infty)\times \partial G,\\
y(0,\cdot) = y_0, & x \in G,
\end{cases}
\end{equation}
where $y_0 \in L^p_{\mathcal{F}_0}(\Omega; L^2(G))$, $\beta_0, \beta_1 \in \mathbb{R}$, $p \ge 1$, and $s \in (0,1]$. We define the operator $A = (-\Delta)^s$ as the spectral fractional power of the Dirichlet Laplacian $-\Delta$. More precisely, let $\{(\lambda_k,\phi_k)\}_{k\ge1}$ denote the eigenpairs of $-\Delta$ with homogeneous Dirichlet boundary conditions. Then $A$ is defined via the spectral decomposition
\[
A\xi = \sum_{k=1}^\infty \lambda_k^s \langle \xi,\phi_k\rangle_H \,\phi_k,
\]
with domain
\[
\mathcal{D}(A) = \left\{ \xi \in L^2(G) \,:\, \sum_{k=1}^\infty \lambda_k^{2s} |\langle \xi,\phi_k\rangle_H|^2 < \infty \right\}.
\]
It follows that $A$ is self-adjoint, positive-definite, and has compact resolvent. Moreover, the eigenfunctions remain unchanged while the eigenvalues become $\lambda_k^s$, that is,
\[
A\phi_k = \lambda_k^s \phi_k, \qquad 0 < \lambda_1^s \le \lambda_2^s \le \cdots,\quad \lambda_k^s \to \infty.
\]

Therefore, equation \eqref{FracSHEquation} fits into our abstract framework. Hence, by Theorem~\ref{pth_stability_abstract}, the solution 
$$y \in L^p_{\mathcal{F}}(\Omega; C([0,\infty); L^2(G)))
\cap L^p_{\mathcal{F}}(\Omega; L^2(0,\infty; D((-\Delta)^{s/2})))$$
of \eqref{FracSHEquation} is $p$-th moment exponentially stable if $
(p-1)\beta_1^2 < 2(\lambda_1^s - \beta_0)$, and by Theorem~\ref{as_stability_abstract}, the solution is almost surely exponentially stable provided that $\beta_1^2 > 2(\beta_0 - \lambda_1^s)$. This example shows that the order $s$ of the fractional Laplacian directly affects the dissipation strength: since the eigenvalues of $(-\Delta)^s$ are $\lambda_k^s$, the operator modifies the spectrum of the Dirichlet Laplacian via the transformation $\lambda_k \mapsto \lambda_k^s$. In particular, when $\lambda_1>1$ (as is the case in typical normalized or sufficiently small domains), decreasing $s$ reduces $\lambda_1^s$, thereby reducing the effective dissipation and leading to more restrictive stability conditions. 

\subsubsection{Stochastic Degenerate Parabolic Equation}
We consider the degenerate equation
\begin{equation}\label{DegenerateSPE}
\begin{cases}
dy - (a(x) y_x)_x\,dt = \beta_0 y\,dt + \beta_1 y\,dW(t), & (t,x) \in (0,\infty)\times (0,1),\\
\begin{cases}
    y(t,0)=0 & \text{ if } K_a \in [0,1),\\
    \displaystyle\lim_{x\to 0^+} a(x) y_x(t,x) =0 & \text{ if } K_a \in [1,2),
\end{cases}  &t \in (0,\infty),\\
y(t,1)=0, & t \in (0,\infty),\\
y(0,\cdot) = y_0, & x \in (0,1),
\end{cases}
\end{equation}
where $y_0 \in L^p_{\mathcal{F}_0}(\Omega; L^2(0,1))$, $\beta_0, \beta_1 \in \mathbb{R}$, and $p \ge 1$. We assume that the diffusion coefficient $a$ satisfies
\[
a \in C([0,1]) \cap C^1((0,1]), \qquad a(x)>0 \ \text{on } (0,1], \qquad a(0)=0,
\]
so that the equation is degenerate at $x=0$ with a degree of degeneracy measured by the parameter $K_a$ defined by
\begin{equation}
    K_a := \sup_{0 < x \leq 1} \frac{x |a'(x)|}{a(x)}.
\end{equation}
Following the standard degenerate framework, we introduce the weighted Sobolev space
\begin{align*}
H_a^1(0,1)
=
\Big\{
u \in L^2(0,1)\;:&\; u\text{ is locally absolutely continuous in } (0,1],\, 
\sqrt{a}\, u_x \in L^2(0,1)
\Big\},
\end{align*}
endowed with the norm $\|u\|_{H_a^1(0,1)}^2 = \|u\|_{L^2(0,1)}^2 + \|\sqrt{a}\, u_x\|_{L^2(0,1)}^2$. We also define the higher-order weighted space
\[
H_a^2(0,1)
=
\left\{
u \in H_a^1(0,1)\;:\; a u_x \in H^1(0,1)
\right\},
\]
endowed with the norm $\|u\|_{H_a^2(0,1)}^2
=
\|u\|_{H_a^1(0,1)}^2 + \|(a u_x)_x\|_{L^2(0,1)}^2$. Define the operator $A$ on $H=L^2(0,1)$ by $Av = -(a v_x)_x, \quad \forall v \in \mathcal{D}(A)$ where 
$$\mathcal{D}(A)=\Bigl\{
v\in H_a^2(0,1): v(1)=0,\ 
v(0)=0 \text{ if } K_a\in[0,1)
\Bigr\}.$$
Then $A$ is a densely defined, self-adjoint, positive-definite operator on $L^2(0,1)$ with compact resolvent, so that the abstract framework applies. Let $\{\lambda_k\}_{k=1}^\infty$ denote the eigenvalues of $A$, satisfying
\[
0 < \lambda_1 \le \lambda_2 \le \cdots \to \infty.
\]

Applying Theorem \ref{pth_stability_abstract}, the solution
$$
y \in L^p_{\mathcal{F}}(\Omega; C([0,\infty); L^2(0,1))) \cap L^p_{\mathcal{F}}(\Omega; L^2(0,\infty; H_a^1(0,1)))
$$
of \eqref{DegenerateSPE} is $p$-th moment exponentially stable if $(p-1)\beta_1^2 < 2(\lambda_1 - \beta_0)$, while Theorem \ref{as_stability_abstract} ensures almost sure exponential stability under the condition $\beta_1^2 > 2(\beta_0 - \lambda_1)$. These conditions reflect the interplay between the diffusion coefficient $a(x)$, the spectral gap $\lambda_1$, the drift $\beta_0$, and the noise intensity $\beta_1$. The degeneracy at $x=0$ is encoded in the weighted space $H_a^1(0,1)$ and influences the operator through its spectrum. In the specific case of power degeneracy $a(x)=x^\alpha$ for $0 \leq \alpha < 2$, the first eigenvalue $\lambda_1$ is characterized by the Rayleigh quotient (see \cite[page 187]{cannarsa2016global}):
\[
\lambda_1 = \inf_{v \in H_a^1(0,1)\setminus\{0\}} 
\frac{\int_0^1 x^\alpha|v_x|^2dx}{\int_0^1 |v|^2dx}.
\]
It follows that a stronger degeneracy (larger $\alpha$) reduces $\lambda_1$, thereby weakening the effective dissipation $\lambda_1 - \beta_0$, which in turn imposes more restrictive stability criteria.


\section{Spectral Decomposition and Abstract Approximation}\label{sec3}
	   As in \cite{Grecksch1996,Kloeden2001}, to facilitate the analysis of \eqref{uncontrolled_for_eq} in an infinite-dimensional vector form, we project the state $y(t)$ onto the orthonormal basis $\{\phi_k\}_{k=1}^{\infty}$ of $H$ consisting of eigenfunctions of $A$. Specifically, we write
		\begin{equation*}
			y(t) = \sum_{k=1}^{\infty} y_k(t) \phi_k,
		\end{equation*}
		where $y_k(t) = \langle y(t), \phi_k \rangle_H$ are the Fourier coefficients.
		
		By taking the inner product of \eqref{uncontrolled_for_eq} with $\phi_k$, we obtain
		\begin{equation*}
        \begin{cases}
			\langle dy(t), \phi_k \rangle_H + \langle Ay(t), \phi_k \rangle_H dt = \beta_0 \langle y(t), \phi_k \rangle_H dt + \beta_1 \langle y(t), \phi_k \rangle_H dW(t),\\
            \langle y(0),\phi_k\rangle_H=\langle y_0,\phi_k\rangle_H=y_0^k, \qquad k = 1, 2, \dots
              \end{cases}
		\end{equation*}
		Using the orthonormality $\langle \phi_i, \phi_k \rangle_H = \delta_{ik}$ and the self-adjointness of $A$, the equation reduces to
		\begin{equation*}
        \begin{cases}
			dy_k(t) + \lambda_k y_k(t) dt = \beta_0 y_k(t) dt + \beta_1 y_k(t) dW(t),\\
            y_k(0)=y_0^k, \qquad k = 1, 2, \dots
              \end{cases}
		\end{equation*}
		This can be written as the following infinite-dimensional system of stochastic ordinary differential equations
		\begin{equation*}
        \begin{cases}
			dy_k(t) = -(\lambda_k - \beta_0) y_k(t) dt + \beta_1 y_k(t) dW(t),\\
            y_k(0)=y_0^k, \qquad k = 1, 2, \dots
              \end{cases}
		\end{equation*}
		
If we denote the coordinate vectors, the spectral matrix, and the identity matrix by
\[
\mathbf{y}(t) = (y_1(t), y_2(t), \dots)^{\top}, \qquad
\mathbf{y}_0 = (y_0^1, y_0^2, \dots)^{\top},
\]
\[
\Lambda = \operatorname{diag}(\lambda_1, \lambda_2, \dots),\qquad I = \operatorname{diag}(1,1,\dots),
\]
then we can write the abstract evolution equation \eqref{uncontrolled_for_eq} in the following compact infinite-dimensional vector form
\begin{equation}\label{eqq39}
\begin{cases}
d\mathbf{y}(t) = -(\Lambda - \beta_0 I) \mathbf{y}(t) dt + \beta_1 \mathbf{y}(t) dW(t),\\
\textbf{y}(0)=\textbf{y}_0.
\end{cases}
\end{equation}
		
\subsection{Spatial Discretization: Spectral Galerkin Approximation}
For the numerical simulation of the stability regimes defined in Theorems \ref{pth_stability_abstract} and \ref{as_stability_abstract}, we proceed as in \cite{yang2025long} to approximate the solution $y$ of \eqref{uncontrolled_for_eq} by its finite dimensional spectral approximation. Let $H_N = \text{span}\{\phi_1, \dots, \phi_N\}$ be the subspace spanned by the first $N$ orthonormal eigenfunctions of $A$. We define the orthogonal projection $P_N: H \to H_N$, so that $P_N y(t) = \sum_{k=1}^N y_k(t) \phi_k$. It is well-established that $P_N y$ approximates $y$ in the $L^p$-norms as $N \to \infty$ (see, e.g., \cite{Canuto2007}).
		
By truncating \eqref{eqq39}, we define the $N$-dimensional coordinate vectors in $\mathbb{R}^N$ by
\[
Y_N(t) := (y_1(t), y_2(t),\dots, y_N(t))^{\top}, 
\qquad 
Y_0^N := (y_0^1, y_0^2, \dots, y_0^N)^{\top}.
\]
Then, the dynamics of this $N$-dimensional spectral Galerkin approximation are governed by the following system of stochastic differential equations
\begin{equation}\label{truncated_system}
\begin{cases}
dY_N(t) = -(\Lambda_N - \beta_0 I_N) Y_N(t) dt + \beta_1 Y_N(t) dW(t),\\
Y_N(0)=Y_0^N,
\end{cases}
\end{equation}
where $\Lambda_N = \operatorname{diag}(\lambda_1, \dots, \lambda_N)$ represents the restriction of the spectral matrix $\Lambda$ to $H_N$, and $I_N = \operatorname{diag}(1, \dots, 1)$ is the identity matrix. 

		
\subsection{Full Discretization}
For the full discretization, we introduce a uniform temporal mesh with step size $\tau > 0$. We define the discrete time points $t_n = n\tau$, for $n \in \mathbb{N}$. Let $Y_n$ denote the discrete approximation of $Y_N(t_n)$. Applying the \textit{implicit} Euler-Maruyama method (see, e.g., \cite[Chapter 8]{Lord2014} or \cite{higham2002strong}) to equation \eqref{truncated_system} leads to the following discrete implicit recursive relation
\begin{equation}\label{discrete_scheme}
\begin{cases}
Y_{n+1} = Y_n - \tau(\Lambda_N - \beta_0 I_N) Y_{n+1} + \beta_1 Y_n \Delta W_n,\\
Y_0=Y_0^N,
\end{cases}
\end{equation}
where $\Delta W_n = W(t_{n+1}) - W(t_n) \sim \mathcal{N}(0, \tau)$ represents independent Brownian increments. To derive the explicit rule for numerical implementation, we isolate the state at the future time step $Y_{n+1}$ and obtain
\begin{equation*}
\left[I_N + \tau(\Lambda_N - \beta_0 I_N)\right]Y_{n+1} = Y_n + \beta_1 Y_n \Delta W_n.
\end{equation*}
In what follows, we work under the conditions \eqref{pthmoexcon} and \eqref{alconst}, which guarantee respectively the $p$-th moment exponential stability and almost sure exponential stability. Consequently, the matrix $M = I_N + \tau(\Lambda_N - \beta_0 I_N)$ is invertible, and we obtain the following discrete explicit recursive relation
\begin{equation}\label{final_numerical_eq}
\begin{cases}
Y_{n+1} = M^{-1} \left( Y_n + \beta_1 Y_n \Delta W_n \right),\\
Y_0=Y_0^N.
\end{cases}
\end{equation}
Hence, we conclude that
\begin{equation}\label{scemeykn}
\begin{cases}
Y_{n+1}^k = \frac{1}{1+\tau(\lambda_k-\beta_0)}
\left( Y_n^k + \beta_1 Y_n^k \Delta W_n \right),  \\
Y_0^k=y_0^k,\qquad k=1,2,\dots,N,
\end{cases}
\end{equation}
where $Y_n^k$ denotes the $k$-th component of the vector $Y_n$. For any $p \ge 1$, the norm of $Y_n$ is given by
\[
\|Y_n\|^p = \left(\sum_{k=1}^N |Y_n^k|^2\right)^{p/2}.
\]

\subsection{Numerical Stability}
In what follows, we show that, under assumptions \eqref{pthmoexcon} and \eqref{alconst}, the fully discrete scheme \eqref{final_numerical_eq} preserves both the $p$-th moment exponential stability and the almost sure exponential stability of \eqref{uncontrolled_for_eq}. To this end, we extend and adapt some ideas from \cite[Section 4]{Higham2007}, originally developed for a class of linear scalar stochastic differential equations. We first establish the \(p\)-th moment exponential stability.
\begin{prop}\label{propso6des}
Assume that condition \eqref{pthmoexcon} holds. Then there exists a small $\tau_0 > 0$ such that for any fixed $\tau \in (0,\tau_0)$, the discrete scheme \eqref{final_numerical_eq} is $p$-th moment exponentially stable. More precisely, there exists a constant \(C>0\), depending only on \(N\) and \(p\) and independent of \(n\), such that
\[
\mathbb{E}(\|Y_n\|^p) \le C e^{-\frac{\mu_p}{2} t_n}\mathbb{E}(\|Y_0\|^p),\quad n\in\mathbb{N},
\]
where $\mu_p$ is given by \eqref{mu_def_abstract}. In particular, the numerical solution $Y_n$ converges to zero exponentially as \(n\to\infty\), in the $p$-th moment sense.
\end{prop}
\begin{proof}
Let $k=1,2,\dots,N$ and $n\in\mathbb{N}$. From \eqref{scemeykn}, we have that
\[
Y_{n+1}^k
=
a_k(\tau)(1+\beta_1 \Delta W_n)Y_n^k,
\qquad
a_k(\tau)=\frac{1}{1+\tau(\lambda_k-\beta_0)},
\]
which implies that
\begin{align}\label{equabe43firt}
Y_n^k
=
Y_0^k \prod_{j=0}^{n-1} a_k(\tau)(1+\beta_1 \Delta W_j).
\end{align}
Since \(Y_0^k\) is \(\mathcal{F}_0\)-measurable, and using the independence of the increments of \(W(\cdot)\), taking expectations on both sides of \eqref{equabe43firt}, we obtain
\begin{equation}\label{eqqforyk0}
\mathbb{E}(|Y_n^k|^p)
=
\mathbb{E}(|Y_0^k|^p)
\Big[|a_k(\tau)|^p \, \mathbb{E}(|1+\beta_1 \Delta W|^p)\Big]^n,
\end{equation}
where \(\Delta W \sim \mathcal{N}(0,\tau)\). Using the scaling property of the normal distribution, we can write
\[
\Delta W = \sqrt{\tau}\,\xi,
\quad \text{where } \xi \sim \mathcal{N}(0,1).
\]
Hence, by applying the Taylor expansion at $0$, we obtain
\[
\mathbb{E}(|1+\beta_1 \Delta W|^p)=
\mathbb{E}\bigl(|1+\beta_1 \sqrt{\tau}\,\xi|^p\bigr)
=
1 + \frac{p(p-1)}{2}\beta_1^2\tau + o(\tau).
\]
Similarly, we have that
\[
|a_k(\tau)|^p
=
1 - p(\lambda_k-\beta_0)\tau + o(\tau).
\]
Hence, we derive
\[
|a_k(\tau)|^p \mathbb{E}(|1+\beta_1 \Delta W|^p)
=
1 -
\Big(p(\lambda_k-\beta_0)
-
\frac{p(p-1)}{2}\beta_1^2\Big) \tau + o(\tau).
\]
Recalling that $\lambda_k \ge \lambda_1$, we obtain 
\[
|a_k(\tau)|^p \mathbb{E}(|1+\beta_1 \Delta W|^p)
\le
1 - \mu_p \tau + o(\tau).
\]
This implies that
\begin{align}\label{ineqq1num}
|a_k(\tau)|^p \mathbb{E}(|1+\beta_1 \Delta W|^p)
\le
1 - \mu_p^\tau \tau,
\qquad \mu_p^\tau:= \mu_p + o(1).
\end{align}
Since $\mu_p^\tau \to \mu_p > 0$ as $\tau \to 0$, there exists a small $\tau_0 > 0$ such that for all $\tau \in (0,\tau_0)$,
\begin{align}\label{ineqq2num}
\mu_p^\tau \ge \frac{\mu_p}{2}.
\end{align}
From \eqref{ineqq1num} and \eqref{ineqq2num}, we deduce that
\begin{align}\label{ineqq3num}
|a_k(\tau)|^p \mathbb{E}(|1+\beta_1 \Delta W|^p)
\le
1 - \frac{\mu_p}{2} \tau.
\end{align}
Combining \eqref{eqqforyk0} and \eqref{ineqq3num}, and using the inequality $1+\alpha \leq e^{\alpha}$, we obtain that
\begin{align}\label{ineqqwithk}
\mathbb{E}(|Y_n^k|^p)
\le
e^{-\frac{\mu_p}{2} t_n}\mathbb{E}(|Y_0^k|^p),\qquad k=1,2,\dots,N.
\end{align}
On the other hand, by equivalence of norms, there exist constants $C_1,C_2>0$ such that
\begin{align}\label{ineqq4num}
C_1\sum_{k=1}^N |Y_n^k|^p\le \|Y_n\|^p \le C_2 \sum_{k=1}^N |Y_n^k|^p.
\end{align}
Summing the inequalities \eqref{ineqqwithk} over \(k = 1,2,\ldots,N\), and using \eqref{ineqq4num}, we conclude that
\[
\mathbb{E}(\|Y_n\|^p)
\le
C e^{-\frac{\mu_p}{2} t_n}\mathbb{E}(\|Y_0\|^p),
\]
where \(C\) depends only on \(N\) and \(p\). This completes the proof of Proposition~\ref{propso6des}.
\end{proof}

We next prove that the fully discrete scheme \eqref{final_numerical_eq} is almost surely exponentially stable.
\begin{prop}\label{proposs7eases}
Assume that condition \eqref{alconst} holds. Then there exists a small $\tau_0 > 0$ such that for any fixed $\tau \in (0,\tau_0)$, the discrete scheme \eqref{final_numerical_eq} is almost surely exponentially stable. More precisely, there exist a finite random variable $\xi(\omega)$ such that
\[
\|Y_n(\omega)\|^p \le \xi(\omega)\, e^{-\frac{\mu_{as}}{4} t_n},
\quad n\in\mathbb{N},\quad \textnormal{a.s.},
\]
where $\mu_{as}$ is given by \eqref{mu_as_def_abstract}. In particular, the numerical solution $Y_n$ converges to zero exponentially as \(n\to\infty\), almost surely.
\end{prop}

\begin{proof}
In this proof, we adopt the same notation as in the proof of Proposition~\ref{propso6des}. Let \(k = 1,2,\dots,N\) and \(n\in\mathbb{N}\). We first observe that
\begin{align}\label{equabe43}
|Y_n^k|^p
= |Y_0^k|^p
\prod_{j=0}^{n-1}
|a_k(\tau)|^p |1+\beta_1 \Delta W_j|^p.
\end{align}
By taking logarithms on both sides of \eqref{equabe43}, we get
\begin{align}\label{ineqqnumm12sec}
\log |Y_n^k|^p
=
\log |Y_0^k|^p + \sum_{j=0}^{n-1} Z_j^k,
\end{align}
where
\begin{align}\label{eqquideinenu1}
Z_j^k
:= \log\Big(|a_k(\tau)|^p |1+\beta_1 \Delta W_j|^p\Big),\qquad \Delta W_j \sim \mathcal{N}(0,\tau),\quad j=0,1,\dots,n-1.
\end{align}
By Taylor expansion at $0$ and the representation $\Delta W_j=\sqrt{\tau}\,\xi$,\; $\xi\sim\mathcal{N}(0,1)$, we obtain 
\begin{align}\label{eqquideinenu2}
\mathbb{E}\big(\log|1+\beta_1 \Delta W_j|\big)
=
\mathbb{E}\big(\log|1+\beta_1 \sqrt{\tau}\,\xi|\big)
=
-\frac{1}{2}\beta_1^2 \tau + o(\tau).
\end{align}
Similarly, we have
\begin{align}\label{eqquideinenu3}
\log |a_k(\tau)|^p
=
-p(\lambda_k-\beta_0)\tau + o(\tau).
\end{align}
Combining \eqref{eqquideinenu1}, \eqref{eqquideinenu2} and \eqref{eqquideinenu3}, we derive that
\begin{align*}
\mathbb{E}(Z_j^k)
&=
-\left[
p(\lambda_k-\beta_0)
+ \frac{p}{2}\beta_1^2
\right]\tau + o(\tau)\\
&\leq -\left[
p(\lambda_1-\beta_0)
+ \frac{p}{2}\beta_1^2
\right]\tau + o(\tau)\\
&= -\mu_{as}^\tau\tau,
\end{align*}
where $\mu_{as}^\tau = \mu_{as} + o(1)$. Since $\mu_{as}^\tau  \to \mu_{as}>0, \quad \text{as } \tau \to 0$, then there exists a small $\tau_0 > 0$ such that for all $\tau \in (0,\tau_0)$, $\mu_{as}^\tau \ge \frac{\mu_{as}}{2}$. Hence, it follows that
\begin{equation}\label{ineexpezkj}
\mathbb{E}(Z_j^k)\leq -\frac{\mu_{as}}{2} \tau.
\end{equation}
Since $(Z_j^k)_j$ are i.i.d. random variables, the Strong Law of Large Numbers implies that
\[
\frac{1}{n}\sum_{j=0}^{n-1} Z_j^k
\longrightarrow \mathbb{E}(Z_0^k),
\quad \text{a.s.}
\]
Recalling \eqref{ineexpezkj}, we conclude that for almost all $\omega \in \Omega$, there exists an integer $n_0^k = n_0^k(\omega) > 0$ such that for all $n \ge n_0^k$,
\[
\frac{1}{n}\sum_{j=0}^{n-1} Z_j^k \le -\frac{\mu_{as}}{4}\tau,
\]
which gives
\begin{align}\label{ineqq1numsec}
\sum_{j=0}^{n-1} Z_j^k \le -\frac{\mu_{as}}{4} t_n.
\end{align}
From \eqref{ineqqnumm12sec} and \eqref{ineqq1numsec}, we deduce that
\[
|Y_n^k|^p \le e^{-\frac{\mu_{as}}{4} t_n} |Y_0^k|^p,\quad \text{for all } n \ge n_0^k.
\]
By defining
\[
\xi_k(\omega):=\max\Big\{|Y_0^k|^p,\max_{0\le n\le n_0^k}\Big(|Y_n^k|^p e^{\frac{\mu_{as}}{4} t_n}\Big)\Big\},
\]
we conclude that
\begin{align}\label{ineqqsum}
|Y_n^k|^p \le \xi_k(\omega) e^{-\frac{\mu_{as}}{4} t_n},\quad \forall n\in\mathbb{N},\quad \text{a.s.}
\end{align}
Summing the inequalities \eqref{ineqqsum} over \(k = 1,2,\dots,N\), we deduce that
\[
\|Y_n\|^p
\le C \sum_{k=1}^N |Y_n^k|^p
\le \xi(\omega)e^{-\frac{\mu_{as}}{4} t_n}, \quad \text{a.s.},
\]
where
\[
\xi(\omega) = C \sum_{k=1}^N \xi_k(\omega).
\]
This completes the proof of Proposition \ref{proposs7eases}.
\end{proof}


\subsection{Numerical Experiments and Stability Verification}
In this subsection, we present numerical experiments to illustrate the previously established stability results, in both $p$-th moment and almost sure senses. To this end, we consider the following fourth-order stochastic parabolic system under hinged boundary conditions
\begin{equation}\label{testequ1} 
    \begin{cases} 
        dy + y_{xxxx} \,dt = \beta_0 y\,dt + \beta_1 y \,dW(t), & (t,x)\in (0,\infty)\times(0,1), \\
        y = y_{xx} = 0, & (t,x)\in(0,\infty)\times\{0,1\}, \\
        y(0,\cdot) =y_0, & x\in(0,1),
    \end{cases} 
\end{equation} 
where the initial condition is prescribed by $y_0(x) := x^4 - 2x^3 + x$.


\subsubsection{Numerical Verification of $p$-th Moment Stability}
In this first experiment, by approximating the expectation using a Monte Carlo method based on an ensemble of $N_s = 50000$ independent sample paths, we analyze the influence of the noise intensity $\beta_1$ on the mean-square stability of the discrete system \eqref{final_numerical_eq} associated with \eqref{testequ1} where we fix $\beta_0 = 1$. We employ a temporal discretization step of $\tau = 10^{-4}$ over a time horizon $T = 0.2$, while the spatial operator is approximated using $N = 100$ spectral modes with eigenvalues $\lambda_k = (k\pi)^4$. 

Figure~\ref{test1_impact_of_noise_intensity_on_mean_square_stability} illustrates the second-moment evolution for $\beta_1 \in \{2, 6, 9\}$.
The trajectories converge to zero, numerically validating Theorem \ref{pth_stability_abstract} which requires the effective dissipation $\lambda_1 - \beta_0$ to dominate the noise intensity. In accordance with \eqref{mu_def_abstract}, increasing $\beta_1$ reduces the decay rate $\mu_p$, leading to a slower convergence toward the equilibrium.


\begin{figure}[H]
    \centering
    \includegraphics[width=0.75\textwidth, keepaspectratio]{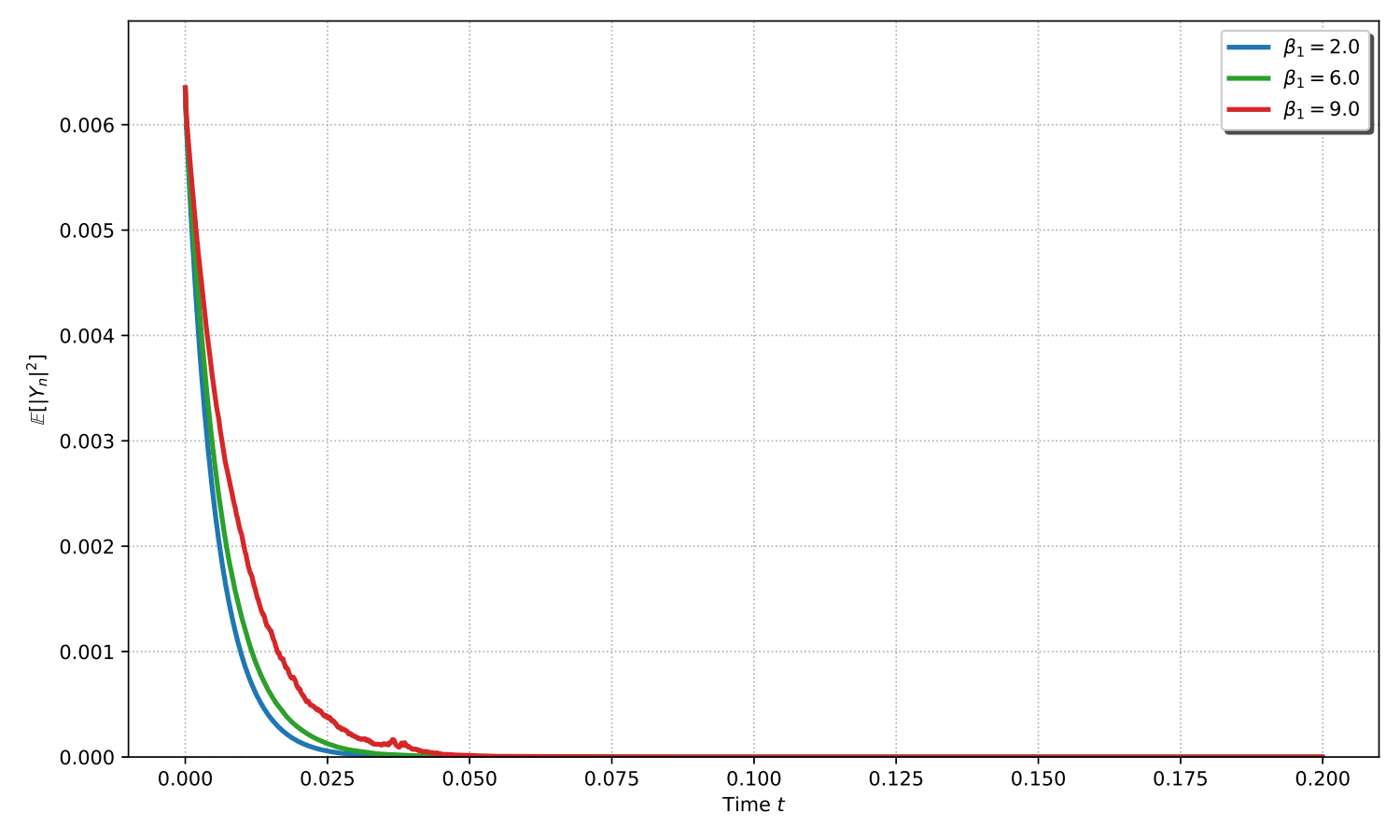}
    \vspace{-0.3cm}
    \caption{Numerical verification of the sensitivity of the mean-square stability to the noise intensity $\beta_1$.}
\label{test1_impact_of_noise_intensity_on_mean_square_stability}
\end{figure}

We next extend the stability analysis to the $p$-th moments for $p \in \{1, 2, 3\}$ to investigate the qualitative shifts in energy dynamics as the moment order increases. As discussed in Remark~\ref{remarksonpmoments}, the quadratic dependence of the stochastic term on $p$ makes the stability criterion significantly more restrictive for higher-order moments. To visualize this sensitivity, we use an ensemble of sample paths $N_s = 100000$ and fix the noise intensity at $\beta_1 = 11$, a value chosen to highlight the transition toward the critical threshold for $p=3$.

Figure~\ref{Test2_Sensitivity_of_moment_stability_to_the_order_p} illustrates the evolution of the normalized $p$-th moments for $p \in \{1, 2, 3\}$. The use of the normalized $p$-th moment $\mathbb{E}[\|Y_n\|^p] / \mathbb{E}[\|Y_0\|^p]$ facilitates a direct comparison of decay rates across different orders. The trajectories numerically validate Theorem~\ref{pth_stability_abstract}; as $p$ increases, the stability condition~\eqref{pthmoexcon} becomes more restrictive, resulting in a diminished decay rate $\mu_p$. For $p=3$, the noise intensity exceeds the sufficient stability threshold, leading to a regime where the quadratic destabilizing term briefly overpowers the linear stabilization; this results in significant transient growth before the restoration of asymptotic decay.

\begin{figure}[H]
    \centering
    \includegraphics[width=0.75\textwidth, keepaspectratio]{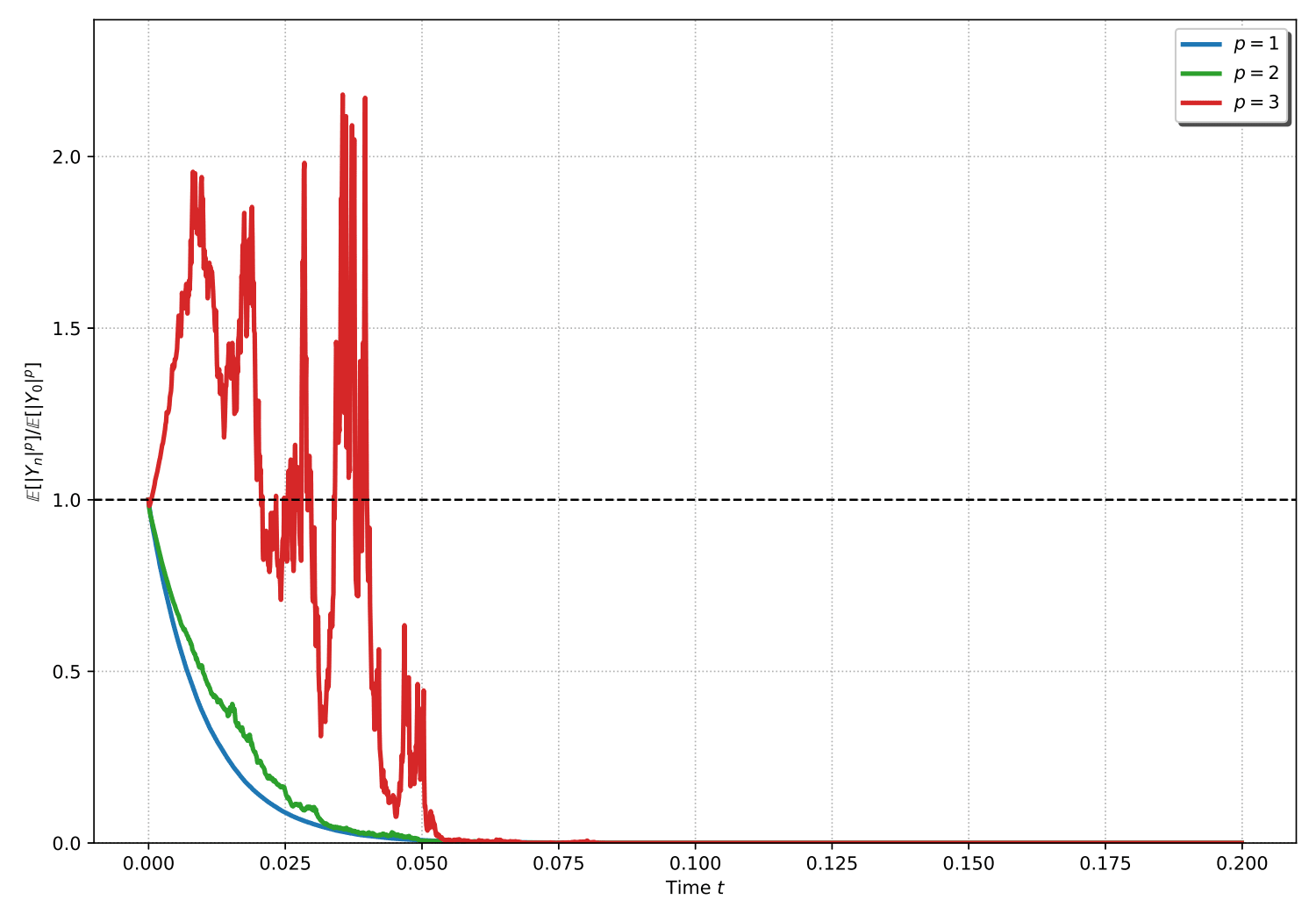}
    \vspace{-0.3cm}
    \caption{Evolution of the normalized $p$-th moments under high noise intensity.}
    \label{Test2_Sensitivity_of_moment_stability_to_the_order_p}
\end{figure}


\subsubsection{Numerical Assessment of Almost Sure Stability}
To verify the pathwise asymptotic behavior alongside numerical evidence for the multiplicative noise-induced stabilization of the stochastic biharmonic heat equation \eqref{testequ1}, we simulate three independent realizations of the discrete solution $Y_n$ using a temporal step $\tau = 10^{-4}$ over a time horizon $T=8$ within the deterministic instability ($\beta_0 = 100 > \lambda_1=\pi^4$) regime. 


Figure \ref{fig:unified_stabilization} numerically validates that the introduction of noise satisfying condition \eqref{alconst} yields almost sure exponential stability. The subplots illustrate a clear sensitivity to the noise intensity $\beta_1$:
\begin{itemize}
    \item At values slightly exceeding the theoretical threshold for almost sure stability, trajectories exhibit persistent stochastic fluctuations with a moderate rate of convergence of the energy $\|Y_n\|^2$ toward the equilibrium.
    \item  Conversely, an increase in noise intensity significantly enhances the stabilization effect, leading to an accelerated decay toward zero.
\end{itemize}

\begin{figure}[H]
    \centering
    \begin{subfigure}[b]{0.75\textwidth} 
        \centering
        \includegraphics[width=\textwidth]{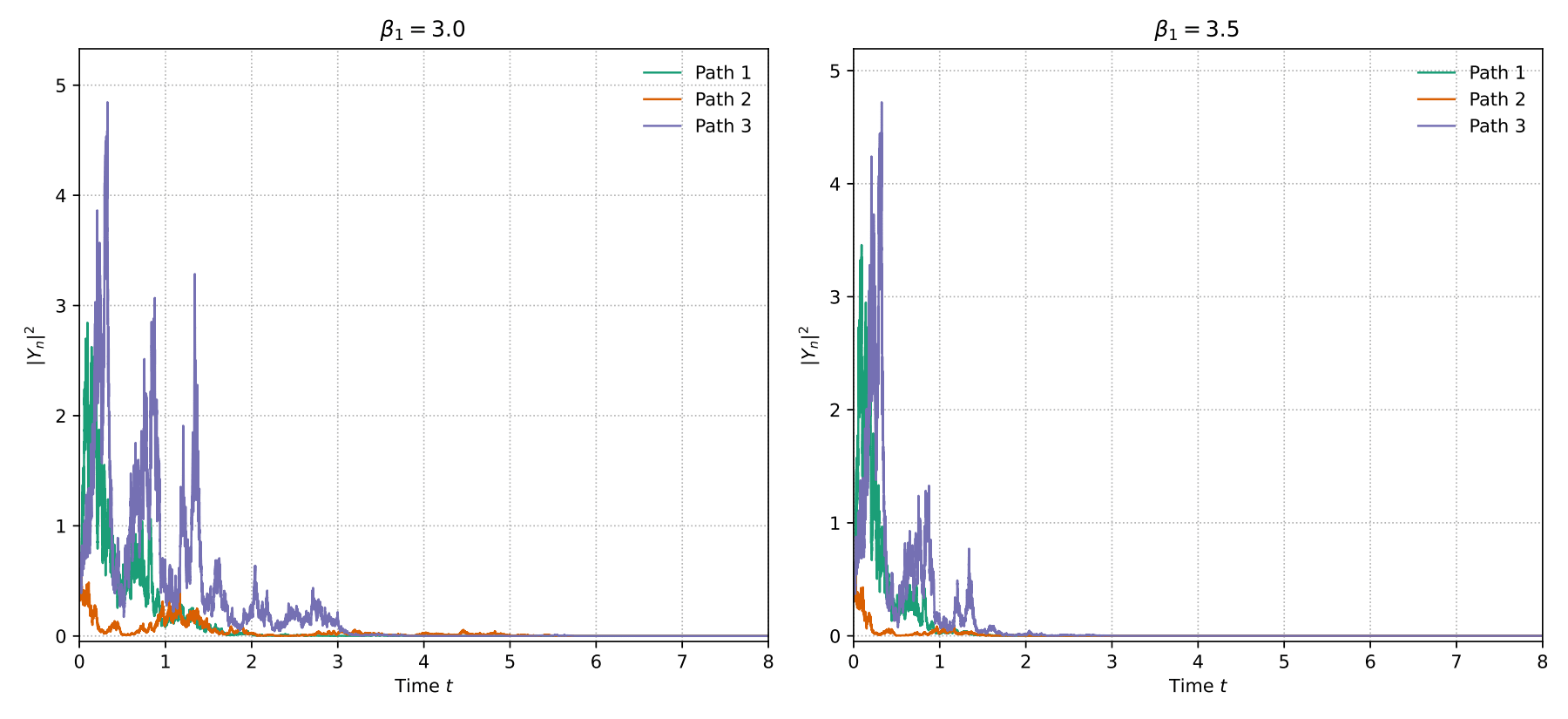}
    \end{subfigure}

    \vspace{0.5cm} 

    \begin{subfigure}[b]{0.75\textwidth}
        \centering
        \includegraphics[width=\textwidth]{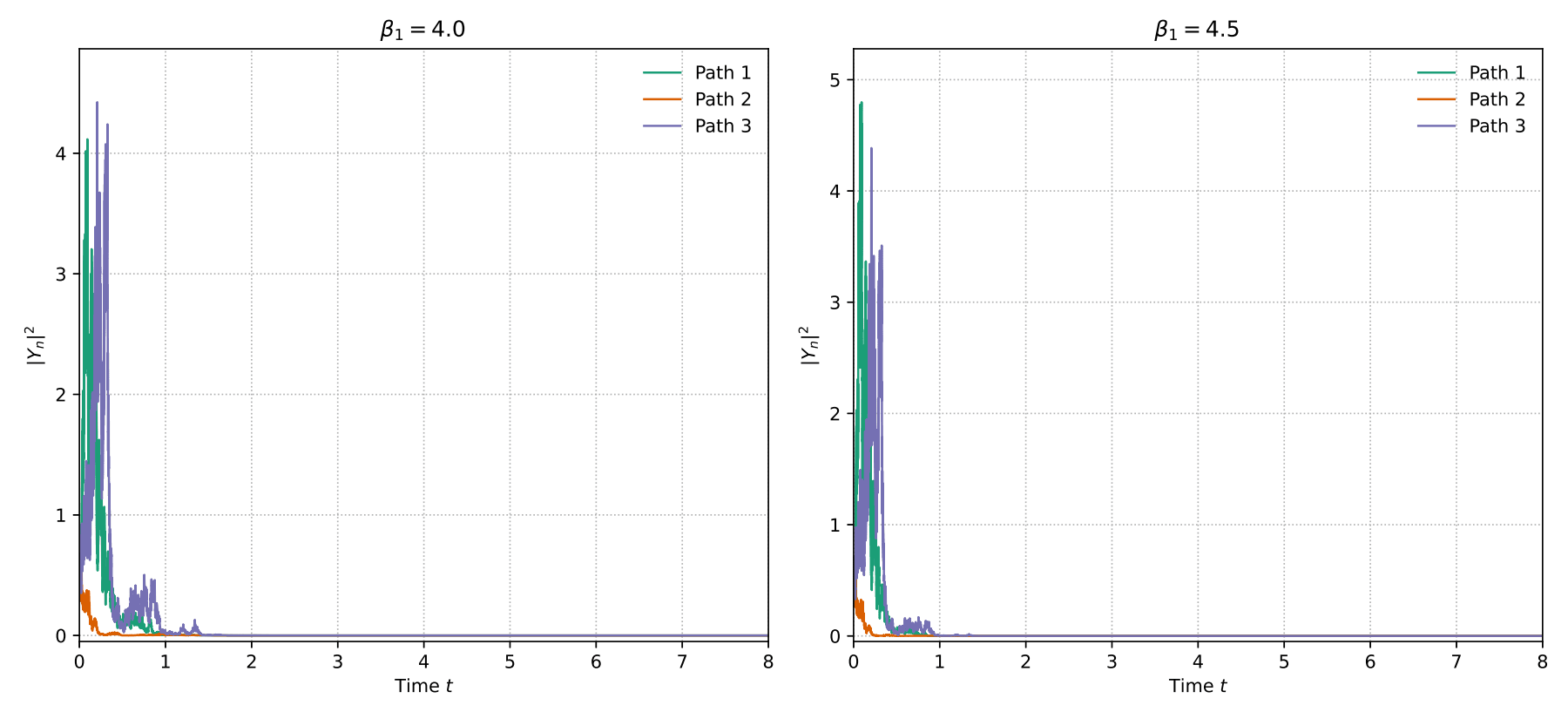}
    \end{subfigure}
    
    \vspace{0.3cm}
    \caption{Evolution of the sample path energy for different noise intensities $\beta_1$.}
    \label{fig:unified_stabilization}
\end{figure}

Unlike the $p$-th moments, Theorem~\ref{as_stability_abstract} predicts that almost sure stability is enhanced as $p$ increases. To clearly visualize the sensitivity of the almost sure decay rate to the power $p$, we set $\beta_0 = 100$ and noise intensity $\beta_1 = 2.7$, satisfying the almost sure stability condition \eqref{alconst}. The numerical simulations were performed with a temporal step $\tau = 10^{-4}$ over $T = 3$, using a single realization to isolate the impact of the power $p$ from stochastic sampling.

The trajectories in Figure \ref{Test4_Pathwise_sensitivity_analysis_for_different_powers} illustrate that the exponential decay rate $\mu_{as}$ is proportional to $p$, resulting in a strictly faster convergence toward equilibrium for higher-order powers.

\begin{figure}[H]
    \centering
    \includegraphics[width=0.75\textwidth, keepaspectratio]{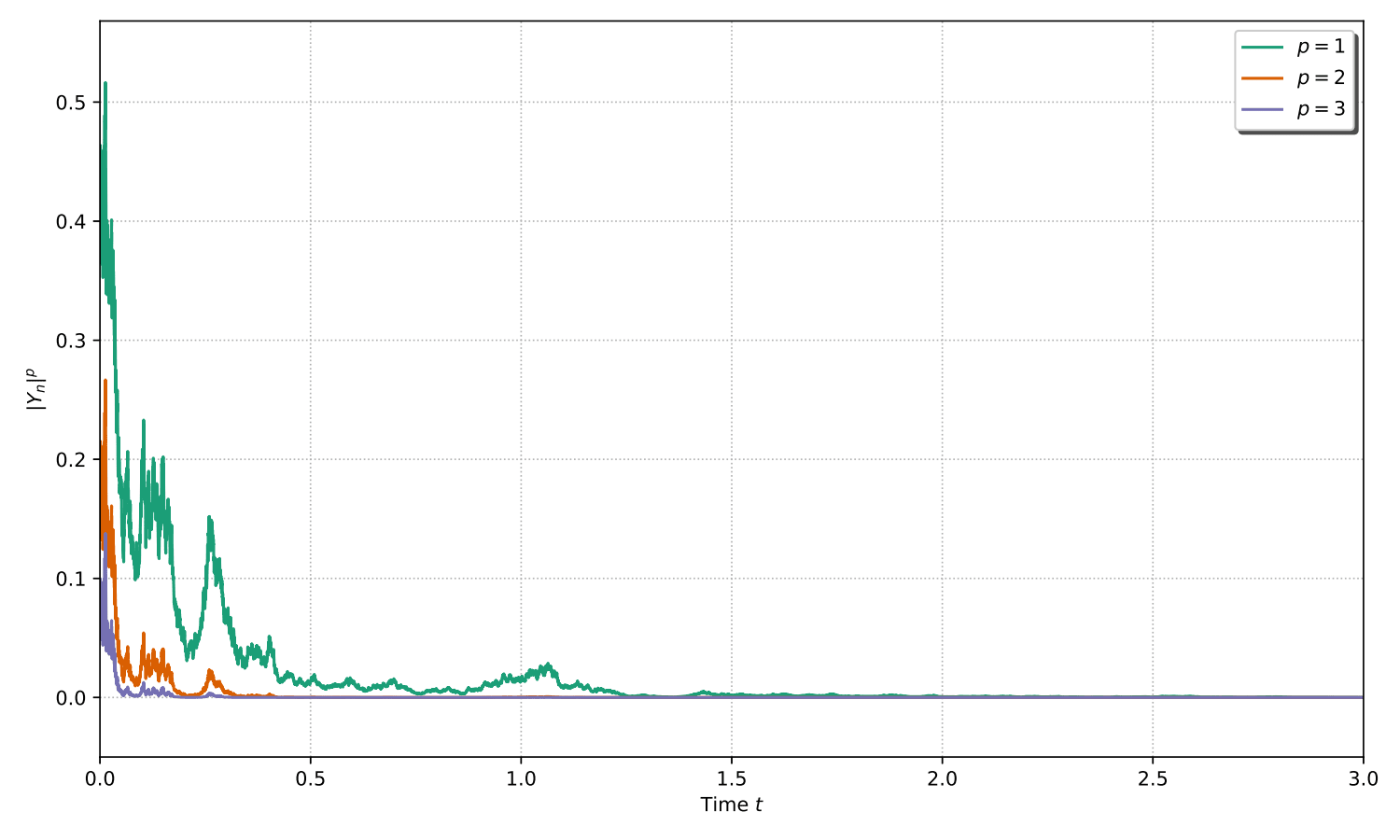}
    \vspace{-0.3cm}
    \caption{Pathwise sensitivity of $\|Y_n\|^p$ for different orders $p$.}
    \label{Test4_Pathwise_sensitivity_analysis_for_different_powers}
\end{figure}

Finally, we investigate the sensitivity of the sample path energy to the noise intensity $\beta_1$ in order to numerically validate the sharpness of the stability condition. By setting $\beta_0 = 97.8$, we observe that when $\beta_1$ does not satisfy the threshold \eqref{alconst}, the energy no longer dissipates and instead exhibits growth rates that intensify as the noise coefficient decreases. These simulations align with the results obtained in \cite[Theorem 2.6]{xie2008moment} for a stochastic heat equation, where it is proven that for deterministic initial conditions, the derived stability condition is both sufficient and necessary. The observed transition from decay to divergence when $\beta_1$ crosses the theoretical threshold confirms that a sufficiently high noise intensity is a fundamental requirement to counteract deterministic instability and ensure the almost sure exponential decay of individual sample paths.

\begin{figure}[H]
    \centering
    \includegraphics[width=0.75\textwidth, keepaspectratio]{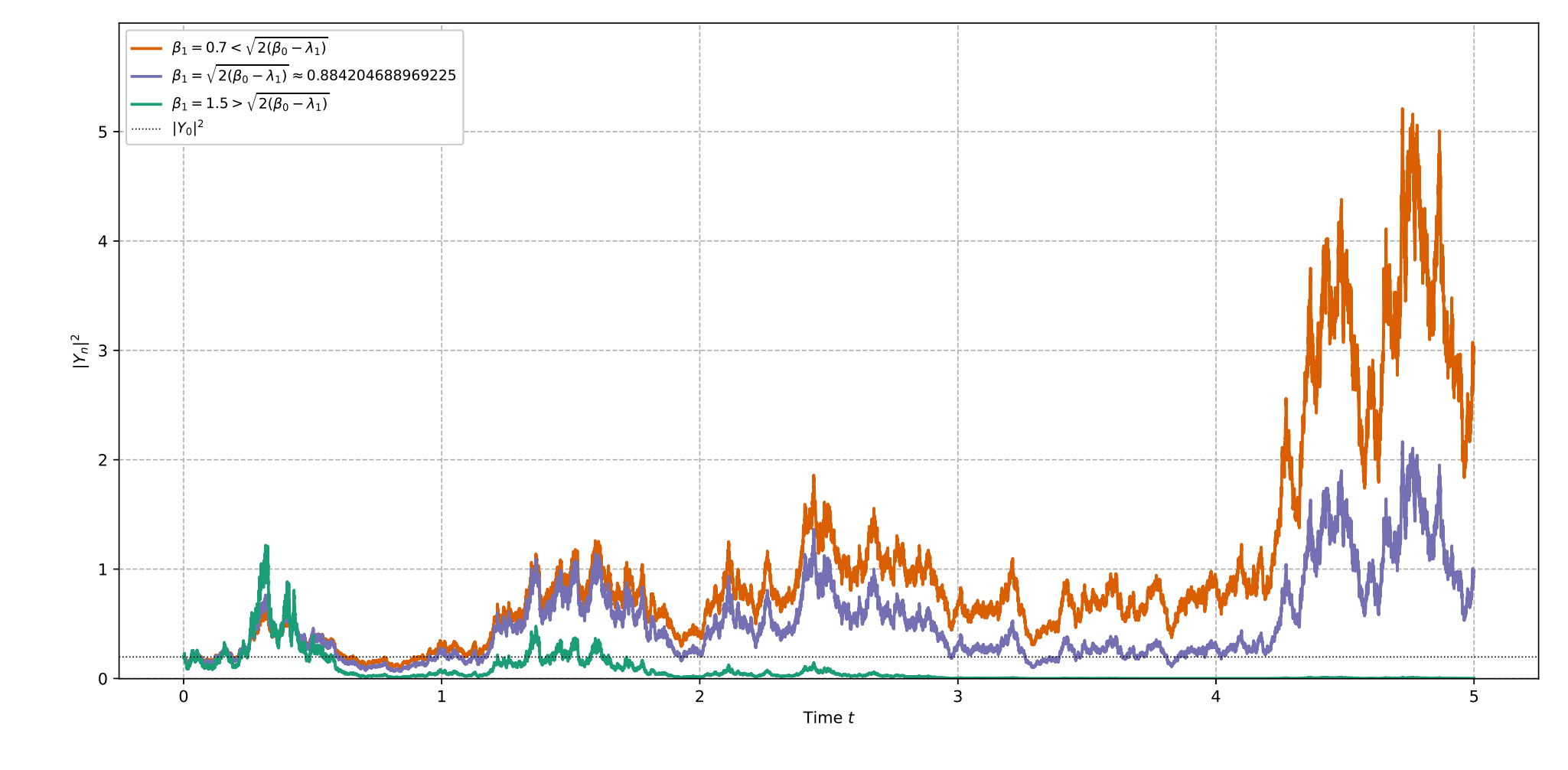}
    \vspace{-0.3cm}
    \caption{Numerical illustration of the sharpness of the almost sure exponential stability condition.}
    \label{Test5_sharpness_of_the_almost_sure_stability_condition}
\end{figure}

\section{Conclusion}

In this paper, we studied the long-time behavior of a class of abstract linear stochastic evolution equations \eqref{uncontrolled_for_eq} in a Hilbert space with multiplicative noise. We derived explicit conditions: condition \eqref{pthmoexcon}, ensuring $p$-th moment exponential stability (Theorem \ref{pth_stability_abstract}), and condition \eqref{alconst}, ensuring almost sure exponential stability (Theorem \ref{as_stability_abstract}). These conditions highlight the role of the principal eigenvalue $\lambda_1$ of the operator $A$, as well as the drift coefficient $\beta_0$ and noise intensity $\beta_1$. We also clarified the relation between these two notions of stability and illustrated how the results apply to several stochastic partial differential equations, including the stochastic heat equation, stochastic biharmonic heat equation, fractional stochastic heat equation, and stochastic degenerate parabolic equations. Finally, we proved that the obtained stability criteria are preserved at the discrete level for the fully discrete approximation of \eqref{uncontrolled_for_eq} obtained via the spectral Galerkin method combined with the implicit Euler--Maruyama scheme. We conclude the paper with numerical experiments applied to a one-dimensional fourth-order stochastic parabolic equation that validate our theoretical results.

\section*{Acknowledgements}
This article is based upon work from the project PRIN2022 D53D23005580006 ``Elliptic and parabolic problems, heat kernel estimates and spectral theory''. The second author is member of the ``Gruppo Nazionale per l’Analisi Matematica, la Probabilità e le loro Applicazioni (GNAMPA)'' of the Istituto Nazionale di Alta Matematica (INdAM).

\end{document}